\documentclass[12pt]{article}
\usepackage{graphicx}
\usepackage{amsmath,amsthm,amssymb,enumerate}
\usepackage{euscript,mathrsfs}
\usepackage[left=2cm,right=2cm,top=3.5cm,bottom=3.5cm]{geometry}
\usepackage{color}
\catcode`\@=11 \@addtoreset{equation}{section}

\catcode`\@=12

\allowdisplaybreaks

\newtheorem{Theorem}{Theorem}[section]
\newtheorem{Proposition}[Theorem]{Proposition}
\newtheorem{Lemma}[Theorem]{Lemma}
\newtheorem{Corollary}[Theorem]{Corollary}

\theoremstyle{definition}
\newtheorem{Definition}[Theorem]{Definition}

\newtheorem{Remark}[Theorem]{Remark}

\newcommand{\bTheorem}[1]{
\begin{Theorem} \label{T#1} }
\newcommand{\eT}{\end{Theorem}}

\newcommand{\bProposition}[1]{
\begin{Proposition} \label{P#1}}
\newcommand{\eP}{\end{Proposition}}

\newcommand{\bLemma}[1]{
\begin{Lemma} \label{L#1} }
\newcommand{\eL}{\end{Lemma}}

\newcommand{\bCorollary}[1]{
\begin{Corollary} \label{C#1} }
\newcommand{\eC}{\end{Corollary}}

\newcommand{\bRemark}[1]{
\begin{Remark} \label{R#1} }
\newcommand{\eR}{\end{Remark}}

\newcommand{\bDefinition}[1]{
\begin{Definition} \label{D#1} }
\newcommand{\eD}{\end{Definition}}

\newcommand{\Del}{\Delta_x}

\newcommand{\prst}{\mathbb{P}}
\newcommand{\p}{\mathbb{P}}

\newcommand{\TN}{\mathcal{T}^N}

\newcommand{\intTN}[1]{\int_{\mathcal{T}^N} #1 \ \dx}
\newcommand{\bfw}{\mathbf{w}}

\newcommand{\bfphi}{\boldsymbol{\varphi}}

\newcommand{\bFormula}[1]{
\begin{equation} \label{#1}}
\newcommand{\eF}{\end{equation}}

\newcommand{\Ov}[1]{\overline{#1}}

\newcommand{\DC}{C^\infty_c}

\newcommand{\vr}{\varrho}

\newcommand{\vu}{\vc{u}}
\newcommand{\vc}[1]{{\bf #1}}

\newcommand{\Div}{{\rm div}_x}
\newcommand{\Grad}{\nabla_x}

\newcommand{\tn}[1]{\mathbb{#1}}
\newcommand{\dx}{\,{\rm d} {x}}

\newcommand{\dt}{\,{\rm d} t }

\newcommand{\dxdt}{\dx \ \dt}

\newcommand{\intT}[1]{ \int_{\TN} #1 \ \dx}

\newcommand{\D}{{\rm d}}
\newcommand{\pas}{\p\text{-a.s.}}
\newcommand{\ep}{\varepsilon}

\newcommand{\expe}[1]{ \mathbb{E} \left[ #1 \right] }
\newcommand{\bas}{(  \Omega, \mathfrak{F}, ( \mathfrak{F}_t )_{t \geq 0} , \prst)}

\definecolor{Cgrey}{rgb}{0.85,0.85,0.85}
\definecolor{Cblue}{rgb}{0.50,0.85,0.85}
\definecolor{Cred}{rgb}{1,0,0}
\definecolor{fancy}{rgb}{0.10,0.85,0.10}

\newcommand\Cbox[2]{%
    \newbox\contentbox%
    \newbox\bkgdbox%
    \setbox\contentbox\hbox to \hsize{%
        \vtop{
            \kern\columnsep
            \hbox to \hsize{%
                \kern\columnsep%
                \advance\hsize by -2\columnsep%
                \setlength{\textwidth}{\hsize}%
                \vbox{
                    \parskip=\baselineskip
                    \parindent=0bp
                    #2
                }%
                \kern\columnsep%
            }%
            \kern\columnsep%
        }%
    }%
    \setbox\bkgdbox\vbox{
        \color{#1}
        \hrule width  \wd\contentbox %
               height \ht\contentbox %
               depth  \dp\contentbox
        \color{black}
    }%
    \wd\bkgdbox=0bp%
    \vbox{\hbox to \hsize{\box\bkgdbox\box\contentbox}}%
    \vskip\baselineskip%
}

\date{}

\begin{document}


\title{On solvability and ill-posedness of the compressible Euler system subject to stochastic forces}

\author{Dominic Breit \and Eduard Feireisl
\thanks{The research of E.F.~leading to these results has received funding from the
European Research Council under the European Union's Seventh
Framework Programme (FP7/2007-2013)/ ERC Grant Agreement
320078. The Institute of Mathematics of the Academy of Sciences of
the Czech Republic is supported by RVO:67985840.}
\and Martina Hofmanov\' a}

\date{\today}

\maketitle

\centerline{Department of Mathematics, Heriot-Watt University}

\centerline{Riccarton Edinburgh EH14 4AS, UK}

\bigskip

\centerline{Institute of Mathematics of the Academy of Sciences of the Czech Republic}

\centerline{\v Zitn\' a 25, CZ-115 67 Praha 1, Czech Republic}

\bigskip

\centerline{Technical University Berlin, Institute of Mathematics}

\centerline{Stra\ss e des 17. Juni 136, 10623 Berlin, Germany}

\bigskip

\begin{abstract}

We consider the (barotropic) Euler system describing the motion of a compressible inviscid fluid driven by a stochastic forcing. Adapting the method of convex integration we show that the initial value problem is  ill-posed in the class of weak (distributional) solutions. Specifically, we find a
sequence $\tau_M \to \infty$ of positive stopping times for which the Euler system admits infinitely many solutions originating from the same initial data.
The solutions are weak in the PDE sense but strong in the probabilistic sense, meaning, they are defined on an {\it a priori} given stochastic basis and adapted to the driving stochastic process.

\end{abstract}


\section{Introduction}
\label{i}

Solutions of
nonlinear systems of conservation laws, including the compressible Euler system discussed in the present paper, are known to develop singularities in a finite time even for smooth
initial data. Weak solutions that can accommodate these singularities provide therefore a suitable framework for studying the behavior of the system
in the long run. A delicate and still largely open question is well-posedness of the associated initial value problem in the class of weak solutions. More precisely, a suitable admissibility criterion is needed to select the physically relevant solution.
The method of \emph{convex integration}, developed in the context of fluid mechanics by De Lellis and Sz\' ekelyhidi \cite{DelSze},
gives rise to several striking results concerning well/ill-posedness of the Cauchy problem for the Euler system and related models of inviscid fluids, see e.g. Chiodaroli  \cite{Chiod}, De~Lellis and Sz\' ekelyhidi \cite{DelSze2}, \cite{DelSze3}. In particular, the barotropic Euler system in 2 and 3 space dimensions is ill-posed in the class of \emph{admissible} entropy solutions (solutions dissipating
energy) even for rather regular initial data, see Chiodaroli, De Lellis and Kreml \cite{ChiDelKre}, Chiodaroli and Kreml \cite{ChiKre}.

In the present paper, we show that this difficulty persists even in the presence of a random forcing. As a model example, we consider
the barotropic Euler system describing the time evolution of the density $\vr$ and the velocity $\vu$ of a compressible fluid:
\begin{eqnarray} \label{i1}
\D \vr + \Div (\vr \vu) \dt &=& 0\\
\label{i2}
\D (\vr \vu) + \Div(\vr \vu \otimes \vu) \dt + \Grad p(\vr) \dt &=& \vr \vc{G}(\vr, \vr \vu)\, \D W,
\end{eqnarray}
where $p=p(\vr)$ is the pressure, and the term $\vr\vc{G}(\vr,\vr\vu) \,\D W$  represents a random volume force acting on the fluid. To avoid technicalities,
we focus on two iconic examples of
forcing, {namely},
\begin{equation} \label{i3}
\vr \vc{G}(\vr, \vr\vu) \,\D W = \vr \vc{G} \,\D W = \vr \sum_{i = 1}^\infty \vc{G}^i\, \D \beta_i,\quad \vc{G}^i = \vc{G}^i(x),
\end{equation}
or
\begin{equation} \label{i4}
\vr \vc{G}(\vr, \vr \vu)\, \D W = \vr \vu\, \D \beta.
\end{equation}
Here $\beta_i = \beta_i(t)$, $\beta = \beta(t)$ are real-valued Wiener  processes whereas the diffusion coefficients
$\vc{G}^i$ are smooth functions depending only on the spatial variable $x$. For the sake of simplicity, we consider periodic boundary conditions, meaning the underlying spatial domain can be identified with a {flat} torus,
\[
\TN = \left( [0,1]|_{\{0,1\}} \right)^N,\ N = 2,3.
\]
{Other boundary conditions, in particular the impermeability of the boundary, could be accommodated at the expense of additional technical difficulties.}

The problem of solvability of the stochastic compressible Euler system (\ref{i1}), (\ref{i2}) is very challenging with only a few results available. In space dimension 1, Berthelin and Vovelle \cite{BV13} proved existence of entropy solutions. These solutions are also weak in the probabilistic sense, that is, the underlying stochastic elements are not known in advance and become part of the solution. The only available results in higher space dimensions concern the local well-posedness of strong solutions. To be more precise, given a sufficiently smooth initial condition
\begin{equation} \label{i5}
\vr(0, \cdot) = \vr_0, \quad \vr \vu(0, \cdot) = (\vr \vu)_0,
\end{equation}
it can be shown that the problem (\ref{i1}), (\ref{i2}), (\ref{i5}) admits a unique local strong solution taking values in the class of Sobolev spaces $W^{m,2}$ of order $m > \frac{N}{2} + 3$. These solutions are strong in both PDE and probabilistic sense, i.e. they are constructed on a given stochastic basis with a given Wiener process. Nevertheless, they exist  (and are unique in terms of the initial data) only up to a strictly positive maximal stopping time $\tau$. Beyond this time that may be finite, the solutions develop singularities and uniqueness is not known. We refer the reader to
\cite{BrFeHo2016} where the stochastic compressible Navier--Stokes system with periodic boundary conditions was treated, and in particular to \cite[Remark 2.10]{BrFeHo2016} for a discussion of the inviscid case.
Let us finally remark that general symmetric hyperbolic systems on the whole space $R^N$ were studied in \cite{JUKim}.

For completeness, let us  mention that (\ref{i4}) may be seen as a ``damping'' term, the regularizing
effect of which in the context of \emph{incompressible} fluids has been recognized by Glatt--Holtz and Vicol \cite{GHVic}, and for a general symmetric
hyperbolic system by Kim \cite{JUKim}. To be more precise, in \cite{JUKim} it was shown that the probability that the strong solution never blows up can be made arbitrarily close to 1 provided the initial condition is sufficiently small. In \cite{GHVic} it was proved that the smallness assumption on the initial condition can be replaced by large intensity of the noise. Besides, in the case of additive noise, which in our setting corresponds to \eqref{i3}, \cite{GHVic} showed global existence of strong solutions to incompressible Euler equations in two dimensions.

\bigskip

Our goal in the present paper is to show that the problem (\ref{i1}), (\ref{i2}) is ill-posed
in the class of weak (distributional) solutions. More precisely, we show that there exists an increasing sequence of strictly positive stopping times
$\tau_M $ with $\tau_M \to \infty$ as $M \to \infty$ a.s., such that problem (\ref{i1}), (\ref{i2}), (\ref{i3}) or (\ref{i4}), (\ref{i5}) admits infinitely many
weak solutions in the time interval $[0, \tau_M \wedge T)$ for any positive $T$.
We emphasize
that \emph{weak} is meant {only} in the PDE sense - the spatial derivatives are understood in the distributional framework - while solutions are \emph{strong} in the probabilistic sense. To be more precise, the stochastic basis together with a driving Wiener process $W$ are given and we construct infinitely many solutions that are stochastic processes adapted to the given filtration.
This is particularly interesting in light of the fact that uniqueness is violated. Indeed, without the knowledge of uniqueness it is typically only possible to construct probabilistically weak solutions that are not adapted to the given Wiener process. This already applies on the level of SDEs, see, for instance, the discussion in \cite[Chapter 5]{KaratzasShreve}.

Formally, both (\ref{i3}) and (\ref{i4}) represent a \emph{multiplicative} noise. Nevertheless, under these assumptions, the system of stochastic PDEs \eqref{i1}, \eqref{i2} may be reduced to a system of PDEs with random coefficients by means of a simple transformation. As a consequence, the stochastic integral no longer appears in the system and deterministic methods can be employed pathwise. Such a semi-deterministic approach was already used in many works,
see for instance Feireisl, Maslowski, Novotn\'y \cite{FeMaNo}, Tornatore and Yashima-Fujita \cite{TorYasA} for the compressible setting, and the seminal paper by Bensoussan and Temam \cite{BenTem} for the incompressible case. However, we point out that in all these references, the nontrivial  issue  of adaptedness of solutions with respect to the underlying stochastic perturbation remained unsolved. Therefore, it was not possible to go back to the original formulation of the problem with a well-defined stochastic It\^o integral.
Even though we employ a similar semi-deterministic approach to \eqref{i1}, \eqref{i2}, \eqref{i3} or \eqref{i4}, we are able to answer affirmatively the question of adaptedness and accordingly the stochastic It\^o integral in the original formulation \eqref{i1}, \eqref{i2} is well-defined.

To be more precise, for both (\ref{i3}) and (\ref{i4}), we rewrite
(\ref{i1}), (\ref{i2}) as an \emph{abstract Euler system} with variable random coefficients in the spirit of \cite{Fei2016}. This relies on the particular structure
of the compressible Euler system and its interplay with stochastic perturbations satisfying  \eqref{i3} or \eqref{i4}.
The resulting problem is then solved by an adaptation of the deterministic method of convex integration developed by De Lellis and Sz\' ekelyhidi \cite{DelSze3}.
The main difficulty is to ensure that the abstract construction based on the concept of subsolutions yields a solution $\vr$, $\vr \vu$ \emph{adapted}
to the noise $W$. This is done by careful analysis of the oscillatory lemma of  De Lellis and Sz\' ekelyhidi \cite{DelSze3},
where {adaptedness} is achieved by a delicate use of the celebrated Ryll--Nardzewski theorem on the existence of a measurable selection of a multivalued
mapping.

The key point is to study a certain non-positive functional $I$ (see Section \ref{sub:var}) defined on an appropriate class of subsolutions (see Section \ref{sec:subsol}) to the abstract Euler system. These subsolutions capture already all the required (probabilistic) properties expected from the solutions.
Similarly to \cite{DelSze3}, the existence of {infinitely many} solutions {to the original problem} is obtained by applying an abstract Baire category argument based on the possibility
of augmenting a given subsolution by rapidly oscillating increments. Determining the amplitude as well as frequency of these oscillatory components at a
given time $t$ requires knowing the behavior of a given subsolution up to the time $t + \delta$, $\delta > 0$. The specific value of $\delta$ is in general a random variable, the value of which depends on the behavior of the noise $W$ in the interval $[t, t + \delta)$. Consequently, it is not adapted
with respect to the natural filtration associated to the noise. The problem can be solved only if $\delta > 0$ is deterministic, specifically if the
solution paths belong to a fixed compact set. To ensure this, we replace $W$ by $W_M(t) = W(t \wedge \tau_M)$, where $\tau_M$ is a family of suitable
stopping times defined in terms of the H\" older norm of $W$. It is exactly this rather technical difficulty that restricts validity of our main result
to the random time interval $[0, \tau_M)$. Note, however, that $\tau_M$ can be made arbitrarily large with probability arbitrarily close to one.

Let us stress that our results apply \emph{mutatis mutandis} to situations when the driving force is given by a more general stochastic process or a deterministic signal of low regularity. Provided a suitable transformation formula to a PDE with random coefficients can be justified, the only  ingredient
is the one required in Section \ref{ss1} for the construction of the corresponding stopping times $\tau_M$. Namely, the trajectories of the driving stochastic process are supposed to be a.s. H\"older continuous for some $a\in (0,1)$. Then existence of infinitely many weak solutions (to the transformed system) adapted to the given stochastic process follows. Whether it is  possible to go back to the original formulation then depends on the particular stochastic process at hand, namely, whether a corresponding stochastic integral can be constructed. If the driving signal is a deterministic H\"older continuous path, the stopping times are not needed and we obtain infinitely many weak solutions (to the transformed system) defined on the full time interval $[0,T]$.

It is important to note that the restriction to the multidimensional case $N=2,3$ is absolutely essential here and the
the variant of the method of convex integration presented below does not work for $N=1$.
Indeed, the method leans on the property of the system to admit oscillatory solutions. As observed in the pioneering work by DiPerna
\cite{DiP1}, \cite{Dip3}, the deterministic counterpart of (\ref{i1}), (\ref{i2}) appended by suitable admissibility conditions gives rise to a solution set that is precompact
in the $L^p$ framework if $N=1$.

\bigskip

To conclude this introductory part, let us summarize the current state of understanding of a compressible flow of an inviscid fluid under stochastic perturbation. Consider a sufficiently smooth initial condition \eqref{i5} and a fixed stochastic basis. On the one hand, it can be shown that there exists a unique local strong solution. However, in view of our result, there exist infinitely many weak solutions emanating from the same initial datum. The very natural question is therefore whether one can compare these two kinds of solutions. In fluid dynamics, it is often possible to establish the so-called weak--strong uniqueness result: strong solutions coincide with weak solutions satisfying a suitable form of  energy inequality. The corresponding result for stochastic compressible Navier--Stokes system was proved in \cite{BFH17}. Consequently, it would be interesting to see whether our weak solutions could be constructed to satisfy an energy inequality. In analogy with the deterministic setting, we know this might be possible only for certain initial data and 
we leave this problem to be addressed in future work.

The paper is organized as follows. In Section \ref{M}, we introduce a proper definition of a weak solution and state our main results.
In Section \ref{T}, the problem is rewritten in a semi-deterministic way that eliminates the explicit presence of stochastic integrals. In Section
\ref{AE}, we rewrite the system as an abstract Euler problem in the spirit of \cite{Fei2016}. Section \ref{CI} is the heart of the paper. Here, the
apparatus of convex integration developed by De Lellis and Sz\' ekelyhidi \cite{DelSze3} is adapted to stochastic framework. The main result is a stochastic
variant of the oscillatory lemma (Lemma \ref{OL1}) proved via Ryll--Nardzewski theorem on measurable selection. The proof of the main result is completed in Section \ref{Inf}.

\section{Problem formulation and main results}
\label{M}

Let $\bas$ be a  probability space with a complete right-continuous filtration $(\mathfrak{F}_t)_{t\geq0}$.
For the sake of simplicity, we restrict ourselves to the case of a single noise, specifically,
\begin{equation} \label{M1}
\vr \vc{G} (\vr, \vr \vu) \D W = \vr \vc{G}(x) \D \beta \quad \mbox{or}\quad \vr \vc{G} (\vr, \vr \vu) \D W =\vr \vu \D \beta,
\end{equation}
where $\beta = \beta(t)$ is a standard Wiener process relative to the filtration $(\mathfrak{F}_t)_{t\geq0}$. In particular, we may correctly define the stochastic integral (in It\^{o}'s sense)
\[
\int_0^\tau \left( \intTN{ \vr \vc{G}(\vr, \vr \vu) \cdot \bfphi } \right) \D W
\]
as soon as the processes
\begin{equation} \label{MMM1-}
t  \mapsto \intTN{ \vr \phi }, \ t \mapsto \intTN{ \vr \vu \cdot \bfphi }
\end{equation}
are $(\mathfrak{F}_t)$-progressively measurable for any smooth (deterministic) test functions $\phi = \phi(x)$ and $\bfphi = \bfphi(x)$.

\begin{Definition} \label{MD1}

We say that $[\vr, \vu, \tau]$ is a weak solution to problem (\ref{i1}), (\ref{i2}), (\ref{i5}) with a stopping time $\tau$ provided
\begin{enumerate}[i)]
\item
$\tau \geq 0$ is an $(\mathfrak{F}_t)$-stopping time;
\item The density $\vr$ is  $(\mathfrak{F}_t)$-adapted and satisfies
\[
\vr \in C([0, \tau); W^{1, \infty}(\TN)), \ \vr > 0 \ \prst\mbox{-a.s};
\]
\item The momentum $\vr\vc{u}$ satisfies $t \mapsto \intTN{ \vr \vu \cdot \bfphi }\in C([0,\tau])$  for any $\bfphi \in \DC(\TN; R^N)$, the stochastic process $t \mapsto \intTN{ \vr \vu \cdot \bfphi }$ is $(\mathfrak{F}_t)$-adapted,
\[
\begin{split}
\vr \vu \in C_{\rm weak}([0, \tau); L^2(\TN; R^N)) \cap L^\infty((0,\tau) \times \TN; R^N)\ \pas;
\end{split}
\]
\item For all $\phi \in \DC(\TN)$ and all $t\geq0$ the following holds $\prst$-a.s.
\begin{equation} \label{MMM1}
\intTN{ \vr(t \wedge \tau , \cdot) \phi } -
\intTN{ \vr_0 \phi } = \int_0^{t \wedge \tau} \intTN{ \vr \vu \cdot \Grad \phi } \dt;
\end{equation}
\item For all $\bfphi \in \DC(\TN,R^N)$ and all $t\geq0$ the following holds $\prst$-a.s.
\begin{equation} \label{MMM2}
\begin{split}
&\intTN{ \vr \vu (t \wedge \tau , \cdot) \cdot \bfphi } -
\intTN{ (\vr \vu)_0 \cdot \bfphi } \\ &= \int_0^{t \wedge \tau} \intTN{ \left[ \vr \vu \otimes \vu : \Grad \bfphi + p(\vr) \Div \bfphi \right] } \dt
+ \int_0^{t \wedge \tau} \left( \intTN{ \vr \vc{G} \cdot \bfphi } \right) \D {W}.
\end{split}
\end{equation}

\end{enumerate}

\end{Definition}

\begin{Remark} \label{AdR1}

The processes (\ref{MMM1-}) are continuous and $(\mathfrak{F}_t)$-adapted; whence progressively measurable. Consequently,
the stochastic integral in (\ref{MMM1}) is correctly defined
as soon as $\vc{G} = \vc{G}(\vr, \vr \vu)$ satisfies (\ref{M1}).

\end{Remark}

We are ready to formulate our main result.

\begin{Theorem} \label{MT1}

Let $T > 0$ and
the initial data $\vr_0, (\vr \vu)_0$ be $\mathfrak{F}_0$-measurable such that
\begin{equation} \label{indata}
\vr_0 \in C^3(\TN), \ (\vr \vu)_0 \in C^3(\TN; R^N), \ \vr_0 > 0 \ \prst\mbox{-a.s.}
\end{equation}
Let the stochastic term satisfy \eqref{M1}, where $\beta$ is a standard Wiener process,
and the coefficient $\vc{G} \in W^{1,\infty}(\TN; R^N)$ is a given deterministic function.
Finally, suppose that the pressure function $p=p(\vr)$ satisfies
$$p\in C^1[0,\infty)\cap C^2(0,\infty),\ p(0)=0.$$
Then there exists a family of $\p$-a.s. strictly positive $(\mathfrak{F}_t)$-stopping times $\tau_M$ satisfying $\tau_{M} \leq \tau_L$ $\p$-a.s. for
$M \leq L$, and
\[
\tau_M \to \infty \ \mbox{as}\ M \to \infty \ \prst\mbox{-a.s.},
\]
such that problem \eqref{i1}, \eqref{i2}, \eqref{i5} admits infinitely many weak solutions with the stopping time
$\tau = \tau_M \wedge T$.

\end{Theorem}

\begin{Remark} \label{MR1}

Solutions obtained in Theorem \ref{MT1} are ``almost global'' in the sense that for any $\ep > 0$, problem (\ref{i1}), (\ref{i2}), (\ref{i5})
admits infinitely many (weak) solutions living on a given time interval $(0,T)$ with probability $1 - \ep$ (choosing $M$ large enough).

\end{Remark}

The rest of the paper is devoted to the proof of Theorem \ref{MT1}.
Let us now summarize the key points of our construction.
For both (\ref{i3}) and (\ref{i4}), we rewrite
(\ref{i1}), (\ref{i2}) as an \emph{abstract Euler system} with variable random coefficients in the spirit of \cite{Fei2016}.
On the set of subsolutions to this system we define the functional
\[
I[\vc{v}] = \expe{ \int_0^T \intTN{ \left[ \frac{1}{2} \frac{|\vc{v} + \vc{h}|^2}{r} - e \right] } \dt }.
\]
It is rather standard to see that $I$ has infinitely many continuity points and that
$I[\vc{v}]=0$ implies that $\vc{v}$ is a solution. The bulk is to show that each continuity point satisfies $I[\vc{v}]=0$ which implies the existence of infinitely many solutions. The latter statement can be shown indirectly by augmenting a given continuity point by rapidly oscillating increments. These increments are obtained by an adaptation of the deterministic method of convex integration developed by De Lellis and Sz\' ekelyhidi \cite{DelSze3}. The main difficulty is to ensure progressive measurability in this construction. Following \cite{DoFeMa} we proceed in three steps:
\begin{itemize}
\item[(i)] Assuming the subsolution under consideration is constant in space-time (but random)
we gain an oscillator sequence which is a random variable itself by the Ryll--Nardzewski theorem on measurable selection. This is first done on the unit interval with density equal to one  (see Lemma \ref{OL1}). A more genral version follows by scaling (see Lemma \ref{OL2}).
\item[(ii)] The construction from (i) can be extend to piecewise constant subsolutions which evaluated at the first time-point of each sub interval. This ensures progressive measurability of the oscillator sequence (see Lemma \ref{OL3}).
\item[(iii)] Finally, we consider the general case of continuous subsolutions  (see Lemma \ref{OL4}). They can be approximated by piecewise constant ones
and we can apply step (ii).
It is important that the modulus of continuity can be controlled. This is where the stopping times in the noise come into play.
\end{itemize}

\section{Transformation to a semi-deterministic setting}
\label{T}

In view of the difficulties mentioned in Section \ref{i}, we are forced to replace the original Wiener
process $\beta$ by a suitable truncation and to rewrite the problem in a semi-deterministic setting.

\subsection{Stopping times}
\label{ss1}

We start by fixing a family $(\tau_M)_{M\in\mathbb{N}}$ of stopping times enjoying the properties claimed in Theorem \ref{MT1}. For a given
$0 < a < \frac{1}{2}$ and the Wiener process $\beta$, $\beta(0) = 0$ $\prst$-a.s., we introduce
\[
O(t) = \sup_{ 0 \leq s \leq t} |\beta(s)| + \sup_{0\leq t_1 \ne t_2 \leq t} \frac{ |\beta(t_1) - \beta(t_2)| }{|t_1 - t_2|^a } \
\mbox{for}\ t > 0,\
O(0) = 0.
\]
{Obviously, $O$ is a non-decreasing stochastic process adapted to $(\mathfrak{F}_t)_{t\geq0}$.}
{Moreover, as $\beta$ is a Wiener process, it follows from the Kolmogorov continuity criterion that}
\[
| \beta(t_1) - \beta (t_2) | \leq B(T, b) |t_1 - t_2|^b = B(T,b) |t_1 - t_2|^{b-a}|t_1 - t_2|^a \ \mbox{whenever}\
0 \leq t_1, t_2 \leq T,
\]
{for any $0 < a < b < \tfrac{1}{2}$, $T > 0$, where $B(T,b)$ is random and finite $\prst$-a.s. In particular, we deduce that
that $O$ is continuous in $[0, \infty)$.}
As a consequence, for $M\in \mathbb{N}$
\[
\tau_M = \inf_{t \geq 0} \left\{ O(t) > M \right\}\wedge T
\]
defines  an $(\mathfrak{F}_t)$-stopping time.
Moreover, $\tau_M \leq \tau_L$ $\p$-a.s. for $M \leq L$, and in particular we get
\[
\tau_M \to \infty \ \mbox{as}\ M \to \infty \ \prst\mbox{-a.s.}
\]
Finally, as $O$ is continuous and $O(0)=0$ $\prst$-a.s., we have that
$\tau_M > 0$ $\prst$-a.s. for all $m\in\mathbb{N}$.

Next, let us introduce the stopped stochastic process
\[
W_M = \beta_M, \ \beta_M(t) = \beta( t \wedge \tau_M) \ \mbox{for}\ t \geq 0.
\]
We recall that for $\tau=\tau_M$, the stochastic integral in \eqref{MMM2} can be rewritten as
$$
\int_0^{t\wedge\tau_M}\left( \intTN{ \vr \vc{G} \cdot \bfphi } \right) \D {W}=\int_0^{t}\left( \intTN{ \vr \vc{G} \cdot \bfphi } \right) \D {W_M}.
$$
From now on, we consider problem (\ref{i1}), (\ref{i2}), (\ref{i5}) with $\beta$ replaced by $\beta_M$. Under these circumstances,
our task reduces
to showing Theorem \ref{MT1} with $\beta = \beta_M$ on the \emph{deterministic} time interval $[0,T]$. Note that the paths of $\beta_M$ are
\emph{uniformly} bounded and \emph{uniformly} H\" older continuous,
\begin{equation} \label{TTT1}
\| \beta_M \|_{C^a[0,T]} \leq M, \ 0<a<\tfrac{1}{2}\ \prst\text{-a.s.}
\end{equation}
This is the essential property we use to construct probabilistically strong solutions, that is, solutions  that are adapted to the given filtration $(\mathfrak{F}_t)_{t\geq0}$ associated to $\beta$.

\subsection{Problem with additive noise}

If the noise is given by (\ref{i3}), we may combine It\^{o}'s calculus with the equation of continuity (\ref{MMM1}) to rewrite the stochastic integral in the form
\[
\int_0^t \left( \intTN{ \vr \vc{G} \cdot \bfphi } \right) \D \beta_M =
 \left( \intTN{ \vr \vc{G} \cdot \bfphi } \right) \beta_M(t)  -
\int_0^t \beta_M(s) \intTN{ \vr \vu \cdot \Grad (\vc{G} \cdot \bfphi) } \ {\rm d}s.
\]
Consequently,
the momentum equation (\ref{i2}) can be formally written as
\begin{equation} \label{T1-}
\D (\vr \vu - \vr \beta_M \vc{G} ) + \Div(\vr \vu \otimes \vu) \dt + \Grad p(\vr) \dt = \beta_M \vc{G} \Div (\vr \vu) \dt,
\end{equation}
where no stochastic integration is necessary. Passing to the weak formulation, our task reduces to finding
$\vr$ and $\vr \vu$ such that
\begin{equation} \label{TT2}
\begin{split}
t &\mapsto \intTN{ \vr \phi }, \ t \mapsto \intTN{\vr \vu \cdot \bfphi }
\ \mbox{continuous and}\ (\mathfrak{F}_t)\mbox{-adapted},\\
\intTN{ \vr(0, \cdot) \phi } &= \intTN{ \vr_0 \phi },\  \intTN{ \vr \vu (0, \cdot) \cdot \bfphi } = \intTN{ (\vr \vu)_0 \cdot \bfphi }\\
& \mbox{for any smooth test functions}\ \phi, \bfphi,
\end{split}
\end{equation}
satisfying
\begin{equation} \label{T1}
\int_0^T
\intT{ \left[ \vr \partial_t \phi + \vr \vu \cdot \Grad \phi \right] } \dt = 0
\end{equation}
for any $\phi \in \DC((0,T) \times \TN)$;
\begin{equation} \label{T2}
\begin{split}
&\int_0^T \intT{ \left[ \left(  \vr \vu - \vr \beta_M \vc{G} \right) \cdot \partial_t \bfphi + \vr \vu \otimes \vu : \Grad \bfphi +
p(\vr) \Div \bfphi\right] } \dt \\
&= \int_0^T \intT{ \left[  \beta_M \vr \vu \cdot \Grad \vc{G} \cdot \bfphi + \beta_M \vr \vu \cdot \Grad \bfphi \cdot \vc{G} \right] } \dt
\end{split}
\end{equation}
for any $\bfphi \in \DC((0,T) \times \TN; R^N)$.

\begin{Remark} \label{TR1}

Problem (\ref{T1}), (\ref{T2}) can be viewed as a system of partial differential equations with random coefficients. We point out that all steps
leading from the original problem (\ref{MMM1}), (\ref{MMM2}) to (\ref{T1}), (\ref{T2}) are reversible as long as
$\vr$, $\vr \vu$ are weakly continuous $(\mathfrak{F}_t)$-adapted and It\^{o}'s calculus applies. In particular, it is enough to solve
(\ref{TT2})--(\ref{T2}).

\end{Remark}

\subsection{Problem with linear multiplicative noise (stochastic ``damping'')}

If the forcing is given by (\ref{i4}), we may again use It\^{o}'s calculus for $0 \leq t \leq \tau_M$ obtaining
\[
\D \exp \left( - \beta_M \right) = - \exp \left( - \beta_M \right) \D \beta_M + \frac{1}{2} \exp \left( -  \beta_M \right) \dt,
\]
and,
\[
\begin{split}
\exp \left( - \beta_M \right) &\left[ \D \left( \intTN{ \vr \vu \cdot \bfphi } \right)  - \left( \intTN{ \vr \vu \cdot \bfphi } \right)  \D \beta_M \right]
\\&= \D \left[ \exp \left( - \beta_M \right) \intTN{\vr \vu \cdot \bfphi }  \right] +  \frac{1}{2}\exp \left( - \beta_M \right) \intTN{ \vr \vu \cdot \bfphi } \dt.
\end{split}
\]
On the other hand, in accordance with (\ref{MMM2}),
\[
\D \left( \intTN{ \vr \vu \cdot \bfphi } \right)  - \left( \intTN{ \vr \vu \cdot \bfphi } \right)  \D \beta_M = \intTN{ \left[
\vr \vu \otimes \vu : \Grad \bfphi + p(\vr) \Div \bfphi \right] } \dt.
\]
We therefore conclude that
\[
\begin{split}
\D \left[ \exp \left( - \beta_M \right) \intTN{\vr \vu \cdot \bfphi }  \right]  &= - \frac{1}{2}\exp \left( - \beta_M \right) \intTN{ \vr \vu \cdot \bfphi } \dt
\\
&\quad+ \exp \left( - \beta_M \right) \intTN{ \left[
\vr \vu \otimes \vu : \Grad \bfphi + p(\vr) \Div \bfphi \right] } \dt.
\end{split}
\]

Similarly to the case of additive noise, we may replace (\ref{MMM1}), (\ref{MMM2}) by a system of partial differential equations with
random coefficients, the weak formulation of which reads
\begin{equation} \label{T3}
\int_0^T
\intT{ \left[ \vr \partial_t \phi + \vr \vu \cdot \Grad \phi \right] } \dt = 0
\end{equation}
for any $\varphi \in \DC((0,T) \times \TN)$;
\begin{align}
\nonumber
0&=\int_0^T \intT{ \Big[ \exp\left(- \beta_M \right) \vr \vu  \cdot \partial_t \bfphi + \exp\left(- \beta_M \right) \vr \vu \otimes \vu : \Grad \bfphi +
\exp\left(- \beta_M \right) p(\vr) \Div \bfphi \Big] } \dt \\
&\quad- \frac{1}{2} \int_0^T \intT{   \exp\left(- \beta_M \right) \vr \vu \cdot \bfphi  } \dt\label{T4}
\end{align}
for any $\bfphi \in \DC((0,T) \times \TN; R^N)$, where $\vr$, $\vr \vu$ are stochastic processes satisfying (\ref{TT2}).

\section{Abstract Euler problem}
\label{AE}

Our next goal is to rewrite the problems (\ref{TT2}), (\ref{T1}), (\ref{T2}) and (\ref{TT2}), (\ref{T3}), (\ref{T4}), respectively,
to fit the abstract framework introduced in \cite{Fei2016}. In addition to (\ref{indata}) we suppose that $\pas$
\begin{equation} \label{indata+}
\| \vr_0 \|_{C^3(\TN)} + \|(\vr \vu)_0 \|_{C^3(\TN; R^N)} + \left\| \vr_0^{-1} \right\|_{C(\TN)} \leq D
\end{equation}
for some deterministic constant $D > 0$. We claim that it is enough to show Theorem \ref{MT1} for the initial data satisfying
(\ref{indata+}). Indeed any initial data $\vr_0$, $(\vr \vu)_0$ satisfying (\ref{indata}) can be written as
\[
[\vr_0 , (\vr \vu)_0 ] = \lim_{D \to \infty}\ [\vr_{0, D}, (\vr \vu)_{0,D} ] \ \prst\mbox{-a.s.}
\]
where
\[
[\vr_{0, D}, (\vr \vu)_{0,D} ](\omega) = \left\{ \begin{array}{l} [\vr_0, (\vr \vu)_0 ](\omega)\ \mbox{if (\ref{indata+}) holds}  \\
{[ 1 , 0 ]} \ \mbox{otherwise.}   \end{array} \right.
\]
Let $[\vr_D, (\vr \vu)_D]$ be the solution emanating from the data $[\vr_{0, D}, (\vr \vu)_{0,D} ]$, the existence of which
is guaranteed by Theorem \ref{MT1}.
{We set}
\[
\Omega_D = \left\{ \omega \in \Omega \ \Big|\ [\vr_0, (\vr \vu)_0](\omega) \ \mbox{satisfies (\ref{indata+})} \right\}.
\]
Note that $\Omega_D$ is $\mathfrak{F}_0$-measurable for any $D>0$. Since
\[
[\vr_0, (\vr \vu)_0] = 1_{\Omega_1} [\vr_1, (\vr \vu)_1 ] + \sum_{D = 2}^{\infty} 1_{\Omega_{D} \setminus \Omega_{D-1}} [\vr_{0,D}, (\vr \vu)_{0,D}],
\]
the desired solution for arbitrary initial data satisfying
(\ref{indata}) can be obtained in the form
\[
[\vr, \vr \vu] = 1_{\Omega_1} [\vr_1, (\vr \vu)_1 ] + \sum_{D = 2}^{\infty} 1_{\Omega_{D} \setminus \Omega_{D-1}} [\vr_D, (\vr \vu)_D].
\]

\subsection{Additive noise}
\label{AEA1}

Going back to (\ref{T1}), (\ref{T2}) we write
\[
\vr \vu - \vr \beta_M \vc{G} = \vc{v} + \vc{V} + \Grad \Psi,
\]
where
\[
\Div \vc{v} = 0, \ \intTN{ \vc{v} } = 0, \ \vc{V} = \vc{V}(t) \in R^N \ \mbox{- a spatially homogeneous function.}
\]

\begin{Remark} \label{AdR2}

Note that $\vc{v} + \vc{V}$ represents the standard Helmholtz projection $\Pi_H$
of $\vr \vu - \vr \beta_M \vc{G}$ onto the space of solenoidal functions.

\end{Remark}

To meet the initial conditions (\ref{i5}), we fix
\[
\vc{v}(0, \cdot) = \Pi_H [(\vr \vu)_0 ] - \frac{1}{|\TN|} \intTN{ (\vr \vu)_0 },\
\vc{V} (0) = \frac{1}{|\TN|} \intTN{ (\vr \vu)_0 }, \ \Grad \Psi (0, \cdot) = \Pi^\perp_H [(\vr \vu)_0].
\]
Accordingly, the equation of continuity (\ref{T1}) reads
\begin{equation} \label{AE1}
\partial_t \vr + \Del \Psi + \beta_M \Div (\vr \vc{G}) = 0, \ \vr(0, \cdot) = \vr_0.
\end{equation}
Given $\Psi$, $\beta_M$, and $\vc{G}$, problem (\ref{AE1}) is uniquely solvable by the method of characteristics. Moreover, as $\beta_M$ satisfies
(\ref{TTT1}) and $\vr_0$ is strictly positive uniform in $\Omega$, we may \emph{fix} the potential $\Psi$ and subsequently the density $\vr$ in such a way that
\begin{equation} \label{AE2}
\begin{split}
\Psi & \in C^2([0,T]; C^3(\TN)) \ \pas,\  \Psi(t, \cdot) \ \mathfrak{F}_0\mbox{-measurable for any}\ t,\\
&\qquad\qquad
\| \Psi \|_{C^2([0,T]; C^3(\TN))} \leq c_M \ \pas,\\
\vr & \in C^1([0,T]; C^1(\TN))\ \pas,\ \vr(0, \cdot) = \vr_0,\ \vr \ (\mathfrak{F}_t)\mbox{-adapted}, \\
&\qquad\qquad
\| \vr \|_{C^1([0,T]; C^1(\TN))} \leq c_M, \ \vr \geq \frac{1}{c_M}\  \pas,
\end{split}
\end{equation}
where $c_M > 0$ is a deterministic constant depending on the stopping parameter $M$. Here, we have also used the extra hypothesis
(\ref{indata+}).

Having fixed $\vr$ and $\Psi$, we compute $\vc{V}$ as the unique solution of the differential equation
\[
\frac{{ \rm d} \vc{V}}{{\rm d}t} = - \frac{1}{|\TN|} \intTN{ \left[ \vr \beta^2_M \Grad \vc{G} \cdot \vc{G} + \beta_M \Grad \vc{G} \cdot \Grad \Psi \right]
},\ \vc{V}(0) = \frac{1}{|\TN|} \intTN{ (\vr \vu)_0 }.
\]
In view of (\ref{AE2}) and the assumption $\vc{G} \in W^{1,\infty}(\TN; R^N)$ we easily deduce that
\begin{equation} \label{AE3}
\vc{V} \in C^1([0,T]; R^N)\ \pas, \vc{V} \ \mbox{is}\ (\mathfrak{F}_t)\mbox{-adapted},\
\| \vc{V} \|_{C^1([0,T]; R^N)} \leq c_M\ \pas
\end{equation}

Thus it remains to find $\vc{v}$ to satisfy (\ref{T2}). It turns out that $\vc{v}$ must be a weak solution of the abstract Euler system
\[
\begin{split}
\partial_t \vc{v} & + \Div \left( \frac{(\vc{v} + \vr \beta_M \vc{G} + \vc{V} + \Grad \Psi) \otimes (\vc{v} +
\vr \beta_M \vc{G} + \vc{V} + \Grad \Psi)}
{\vr} \right) \\ &= - \Grad p(\vr) - \partial_t \Grad \Psi+\beta_M \Div \left( \vr \beta_M\vc{G} + \Grad \Psi \right) \vc{G}
- \frac{1}{|\TN|} \intTN{ \beta_M \Div \left( \vr \beta_M\vc{G} + \Grad \Psi \right) \vc{G} },
\end{split}
\]
\[
\Div \vc{v} = 0, \ \vc{v}(0, \cdot) = \vc{v}_0 = \Pi_H [(\vr \vu)_0 ] - \frac{1}{|\TN|} \intTN{ (\vr \vu)_0 }.
\]

Finally, we solve the elliptic system
\begin{equation} \label{AE4}
\begin{split}
&\Div \left[ \Grad \vc{m} + \Grad^t \vc{m} - \frac{2}{N} \Div \vc{m} \mathbb{I} \right]\\
&= \Grad p(\vr)  + \partial_t \Grad \Psi - \beta_M \Div \left( \vr \beta_M\vc{G} + \Grad \Psi \right) \vc{G}
+ \frac{1}{|\TN|} \intTN{ \beta_M \Div \left( \vr \beta_M\vc{G} + \Grad \Psi \right) \vc{G} }.
\end{split}
\end{equation}
Note that (\ref{AE4}) admits a unique solution as the right-hand side is a function of zero mean. Consequently, setting
\[
r = \vr, \ \vc{h} = \vr \beta_M \vc{G} + \vc{V} + \Grad \Psi, \ \mathbb{M} = \Grad \vc{m} + \Grad^t \vc{m} - \frac{2}{N} \Div \vc{m} \mathbb{I}
\]
we may rewrite the problem in a concise form:
\begin{equation} \label{AE5}
\partial_t \vc{v} + \Div \left[ \frac{ (\vc{v} + \vc{h}) \otimes (\vc{v} + \vc{h}) }{r} + \mathbb{M} \right] = 0,
\ \Div \vc{v} = 0, \ \vc{v}(0, \cdot) = \vc{v}_0,
\end{equation}
where
\begin{equation} \label{AE6}
\begin{split}
\vc{v}_0 &\in C^1(\TN; R^N)\ \pas, \ \Div \vc{v}_0 = 0, \ \intTN{ \vc{v}_0 } = 0, \ \vc{v}_0\ \mbox{is} \ \mathfrak{F}_0\mbox{-measurable},\\
&\qquad\qquad
\| \vc{v}_0 \|_{C^1(\TN; R^N)} \leq c_M\ \pas,
\\
\vc{h} &\in C^a ([0,T]; C^1(\TN; R^N))\ \pas, \ \vc{h}\ \mbox{is} \ (\mathfrak{F}_t)\mbox{-adapted},\\
&\qquad\qquad
\| \vc{h} \|_{C^a ([0,T]; C^1(\TN; R^N))} \leq c_M\ \pas, \\
r &\in C^a ([0,T]; C^1(\TN))\ \pas, \ r\ \mbox{is} \ (\mathfrak{F}_t)\mbox{-adapted},\\
&\qquad\qquad
\| r \|_{C^a ([0,T]; C^1(\TN))} \leq c_M, \ \frac{1}{r} \geq \frac{1}{c_M}\ \pas,\\
\mathbb{M} &\in C^a ([0,T]; C^1(\TN; R^{N \times N}_{0, {\rm sym}}))\ \pas, \ \mathbb{M}\ \mbox{is} \ (\mathfrak{F}_t)\mbox{-adapted},\\
&\qquad\qquad
\| \mathbb{M} \|_{C^a ([0,T]; C^1(\TN; R^{N \times N}_{0, {\rm sym}} ))} \leq c_M\ \pas,
\end{split}
\end{equation}
are given data.

\subsection{Multiplicative noise}

Mimicking the steps of the previous section we write
\[
\exp(-\beta_M) \vr \vu = \vc{v} + \vc{V} + \Grad \Psi
\]
in (\ref{T4}),
where
\[
\Div \vc{v} = 0, \ \intTN{ \vc{v} } = 0, \ \vc{V} = \vc{V}(t) \in R^N \ \mbox{is a spatially homogeneous function,}
\]
and
\[
\vc{v}(0, \cdot) = \Pi_H [(\vr \vu)_0 ] - \frac{1}{|\TN|} \intTN{ (\vr \vu)_0 },\
\vc{V} (0) = \frac{1}{|\TN|} \intTN{ (\vr \vu)_0 }, \ \Grad \Psi (0, \cdot) = \Pi^\perp_H [(\vr \vu)_0].
\]
Accordingly, the equation of continuity reads
\begin{equation} \label{AE7}
\partial_t \vr + \Div \left( \exp(\beta_M) \Grad \Psi \right) = 0, \ \vr(0, \cdot) = \vr_0.
\end{equation}

Next, we fix $\vc{V}$ as the unique solution of
\[
\frac{{\rm d}\vc{V}}{\dt} + \frac{1}{2} \vc{V} = 0, \ \vc{V}(0) = \frac{1}{|\TN|} \intTN{ (\vr \vu)_0 }.
\]
Accordingly, the momentum equation can be written as
\begin{align} \nonumber
\partial_t \vc{v} &+  \exp( \beta_M) \left[ \Div \frac{\big( \vc{v} + \vc{V} + \Grad \Psi\big) \otimes \big(\vc{v} + \vc{V} + \Grad \Psi\big) }{\vr} \right]+ \exp(- \beta_M) \Grad p(\vr) + \partial_t \Grad \Psi + \frac{1}{2} \Grad \Psi  \\ &= - \frac{1}{2} \vc{v} ,\
\Div \vc{v} = 0, \ \vc{v}_0 = \Pi_H[(\vr \vu)_0] - \frac{1}{|\TN|}\int_{\TN} (\vr \vu)_0 \dx.\label{A9b}
\end{align}
Similarly to the above, we can fix $\vr$, $\Psi$ to satisfy (\ref{AE7}) together with (\ref{AE2}).

Finally, seeing that $\intTN{ \vc{v} } = 0$, we may solve an analogue to the elliptic system (\ref{AE4}), namely,
\begin{equation} \label{AE8}
\begin{split}
\Div &\left[ \Grad \vc{m} + \Grad^t \vc{m} - \frac{2}{N} \Div \vc{m} \mathbb{I} \right] \\
&= \exp (-\beta_M) \Grad p(\vr) + \partial_t \Grad \Psi + \frac{1}{2} \Grad \Psi + \frac{1}{2} \vc{v}.
\end{split}
\end{equation}
Note that, in contrast with (\ref{AE4}), the solution $\vc{m} = \vc{m}[\vc{v}]$ depends on $\vc{v}$.

Similarly to (\ref{AE5}) we can write the final problem (setting $\vc{h}=\vc{V} + \Grad \Psi$ and $r=\vr$):
\begin{equation} \label{AE9}
\partial_t \vc{v} + \Div \left[ \frac{ (\vc{v} + \vc{h}) \otimes (\vc{v} + \vc{h}) }{r} + \mathbb{M}[\vc{v}] \right] = 0,
\ \Div \vc{v} = 0, \ \vc{v}(0, \cdot) = \vc{v}_0,
\end{equation}
where
\begin{align}
\vc{v}_0 &\in C^1(\TN; R^N)\ \pas, \ \Div \vc{v}_0 = 0, \ \intTN{ \vc{v}_0 } = 0, \ \vc{v}_0\ \mbox{is}  \ \mathfrak{F}_0\mbox{-measurable},\\
&\qquad\qquad
\| \vc{v}_0 \|_{C^1(\TN; R^N)} \leq c_M\ \pas,\nonumber
\\
\vc{h} &\in C^a ([0,T]; C^1(\TN; R^N))\ \pas, \ \vc{h}\ \mbox{is}  \ (\mathfrak{F}_t)\mbox{-adapted},\\
&\qquad\qquad
\| \vc{h} \|_{C^a ([0,T]; C^1(\TN; R^N))} \leq c_M\ \pas, \label{AE10}\\
r &\in C^a ([0,T]; C^1(\TN))\ \pas, \ r\ \mbox{is}  \ (\mathfrak{F}_t)\mbox{-adapted},\\
&\qquad\qquad
\| r \|_{C^a ([0,T]; C^1(\TN)} \leq c_M, \ \frac{1}{r} \geq \frac{1}{c_M}\ \pas,\nonumber
\end{align}
and
\[
\mathbb{M} = \mathbb{M}[\vc{v}] = \Grad \vc{m} + \Grad^t \vc{m} - \frac{2}{N} \Div \vc{m} \mathbb{I}
\]
is the unique solution of the elliptic system (\ref{AE8}).

\begin{Remark}\label{AdR3}

Note that $\vc{h}$ is actually more regular than in Section \ref{AEA1}.

\end{Remark}

\section{Convex integration}
\label{CI}

Problems (\ref{AE5}) and (\ref{AE9}) can be solved pathwise using the method of De Lellis and Sz\' ekelyhidi \cite{DelSze3}, with the necessary modifications
developed in \cite{Fei2016}. In such a way, we would obtain the existence of (infinitely many) solutions in the semi-deterministic spirit introduced by
Bensoussan and Temam \cite{BenTem}. More specifically, solutions obtained this way would be random variables, meaning $\mathfrak{F}$-measurable but not necessarily $(\mathfrak{F}_t)$-adapted (progressively measurable). Obviously, such a semi-deterministic result would hold without any restriction imposed by the stopping times. Progressive measurability of $\vr$, $\vr \vu$ claimed in Theorem \ref{MT1} represents a non-trivial issue that requires
a careful revisiting of the method of convex integration presented in \cite{DelSze3}. The main ingredient is a stochastic variant of the so-called \emph{oscillatory lemma} shown in the present section.

\begin{Definition} \label{DC1}

Let $G: \Omega \to X$ be a (Borelian) random variable ranging in a topological space $X$. We say that $G$ has a compact range in $X$ if there
is a {(deterministic)} compact set $\mathcal{K} \subset X$ such that $G \in \mathcal{K}$ a.s.

\end{Definition}

\subsection{Geometric setting}

Let $R^{N\times N}_{{\rm sym}}$ denote the space of symmetric $N \times N$ matrices and let $R^{N\times N}_{0,{\rm sym}}$ be its subspace of
traceless matrices. Following the ansatz of \cite[Lemma 3]{DelSze3} we
consider the set
\[
\mathcal{S}[e] = \left\{ [\vc{w}, \mathbb{H}] \ \Big| \ \vc{w} \in R^N, \ \mathbb{H} \in R^{N \times N}_{0, {\rm sym}},\
\frac{N}{2} \lambda_{\rm max} \left[ \vc{w} \otimes \vc{w} - \mathbb{H} \right] < e \right\},
\]
where $\lambda_{\rm max}[ \mathbb{A}]$ denotes the maximal eigenvalue of a symmetric matrix $\mathbb{A}$.
Thanks to the algebraic inequality
\begin{equation} \label{AlIn}
\frac{N}{2} \lambda_{\rm max} \left[ \vc{w} \otimes \vc{w} - \mathbb{H} \right] \geq \frac{1}{2} |\vc{w}|^2,\ \mathbb{H} \in R^{N \times N}_{0, {\rm sym}},
\end{equation}
$\mathcal{S}[e] \ne \emptyset$ only if $e > 0$. {In addition, we have}
\begin{equation} \label{AlIn2}
\frac{N-1}{2} \lambda_{\rm max} \left[ \vc{w} \otimes \vc{w} - \mathbb{H} \right] \geq \frac{1}{2} |\tn{H}|^2,\ \vc{w} \in R^N,
\end{equation}
see \cite[Lemma 3 iii)]{DelSze3}.
{Thus, for given $e > 0$}, $\mathcal{S}[e]$ is a convex open and bounded subset
of $R^N \times R^{N \times N}_{0, {\rm sym}}$. Moreover, as shown in \cite{DelSze3},
\[
\partial \mathcal{S}[e] = \left\{ \left[ \vc{a}, \ \vc{a} \otimes \vc{a} - \frac{1}{N} |\vc{a}|^2 \mathbb{I}\right]\ \Big| \
\frac{1}{2} |\vc{a}|^2 = e \right\}.
\]

{De Lellis and Sz\' ekelyhidi \cite[Lemma 6]{DelSze3} proved the following result. Given $e > 0$ and
$[\vc{w}, \mathbb{H}] \in \mathcal{S}[e]$, there exist $\vc{a}, \vc{b} \in R^N$ enjoying the following properties:}
\begin{itemize}
\item we have that
\begin{equation} \label{c1A}
\frac{1}{2}|\vc{a}|^2 = \frac{1}{2}|\vc{b}|^2 = e;
\end{equation}
\item
{there exists $L \geq 0$ such that}
for $\vc{s} = \vc{a} - \vc{b}$, $\mathbb{M} = \vc{a} \otimes \vc{a} - \vc{b}  \otimes \vc{b}$, we have
\begin{equation} \label{c2A}
\begin{split}
[ \vc{w} + \lambda \vc{s}, \mathbb{H} + \lambda \mathbb{M} ] &\in \mathcal{S}[e] ,\\
{\rm dist} \left[ [ \vc{w} + \lambda \vc{s}, \mathbb{H} + \lambda \mathbb{M} ]; \partial \mathcal{S}[e] \right]  &\geq \frac{1}{2}
{\rm dist} \left[ [ \vc{w}, \mathbb{H} ]; \partial \mathcal{S}[e] \right]
\end{split}
\end{equation}
{for all $\lambda \in [-L,L]$};
\item
{
there is a universal constant $c(N)$ depending only on the dimension such that}
\begin{equation} \label{c3A}
L |\vc{s}| \geq c(N) \frac{1}{\sqrt{e}} \left( e - \frac{1}{2} |\vc{w}|^2 \right);
\end{equation}
\item we have that
\begin{equation} \label{C3AB}
\left| \vc{a} \pm \vc{b} \right| \geq \chi \left( {\rm dist} \left[ [ \vc{w}, \mathbb{H} ]; \partial \mathcal{S}[e] \right] \right),
\end{equation}
{where $\chi$ is positive for positive arguments.}
\end{itemize}

{
Motivated by this result,}
we consider a set-valued mapping
\[
\mathcal{F} : (0, \infty) \times R^N \times R^{N \times N}_{0, {\rm sym}} \to 2^{R^N \times R^N}
\]
determined by the following properties:
\begin{enumerate}
\item whenever
$[\vc{w}, \mathbb{H}] \notin \Ov{ \mathcal{S}[e] }$ we have
\begin{equation} \label{c1-}
\mathcal{F}(e, \vc{w}, \mathbb{H} ) = \left\{ [\vc{w}, \vc{w}] \right\};
\end{equation}
\item
If $[\vc{w}, \mathbb{H}] \in \Ov{ \mathcal{S}[e]}$, then
$[\vc{a}, \vc{b}] \in \mathcal{F}(e, \vc{w}, \mathbb{H})$ if {and only if}:
\begin{itemize}
\item we have that
\begin{equation} \label{c1}
\frac{1}{2}|\vc{a}|^2 = \frac{1}{2}|\vc{b}|^2 = e;
\end{equation}
\item
there exists $L \geq 0$ such that
for $\vc{s} = \vc{a} - \vc{b}$, $\mathbb{M} = \vc{a} \otimes \vc{a} - \vc{b}  \otimes \vc{b}$, we have
\begin{equation} \label{c2}
\begin{split}
[ \vc{w} + \lambda \vc{s}, \mathbb{H} + \lambda \mathbb{M} ] &\in \mathcal{S}[e] ,\\
{\rm dist} \left[ [ \vc{w} + \lambda \vc{s}, \mathbb{H} + \lambda \mathbb{M} ]; \partial \mathcal{S}[e] \right]  &\geq \frac{1}{2}
{\rm dist} \left[ [ \vc{w}, \mathbb{H} ]; \partial \mathcal{S}[e] \right]
\end{split}
\end{equation}
for all $\lambda \in [-L,L]$;
\item we have that
\begin{equation} \label{c3}
L |\vc{s}| \geq c(N) \frac{1}{\sqrt{e}} \left( e - \frac{1}{2} |\vc{w}|^2 \right),
\end{equation}
{where $c(N)$ is the universal constant from (\ref{c3A});}
\begin{equation} \label{c3ab}
\left| \vc{a} \pm \vc{b} \right| \geq \chi \left( {\rm dist} \left[ [ \vc{w}, \mathbb{H} ]; \partial \mathcal{S}[e] \right] \right),
\end{equation}
{where $\chi$ has been introduced in (\ref{C3AB}).}
\end{itemize}

\end{enumerate}

{Basic properties of $\mathcal{F}$ are summarized in the following lemma.}

\begin{Lemma} \label{L1}

{
For any $(e,\bfw,\mathbb H)\in (0, \infty) \times R^N \times R^{N \times N}_{0, {\rm sym}}$ the set $\mathcal{F}(e, \vc{w}, \mathbb{H})$} is non-empty, closed, and contained in a compact set, the size of which depends only on $e$ and
$|\vc{w}|$.
Moreover, the mapping
\[
\mathcal{F} : (0, \infty) \times R^N \times R^{N \times N}_{0, {\rm sym}} \to 2^{R^N \times R^N}
\]
has closed graph.

\end{Lemma}

\begin{proof}

As shown \cite[Lemma 6]{DelSze3}, the set $\mathcal{F}(e, \vc{w}, \mathbb{H})$
is non-empty for any $[\vc{w}, \mathbb{H}] \in \mathcal{S}[e]$ for a certain universal constant $c(N)$. If
$[\vc{w}, \mathbb{H}] \in \partial \mathcal{S}[e]$, then
\[
\frac{1}{2} |\vc{w}|^2 = e,
\]
and, consequently, $\mathcal{F}(e,\vc{w}, \mathbb{H})$ contains at least the trivial point $[\vc{w},\vc{w}]$. Obviously,
$\mathcal{F}(e,\vc{w}, \mathbb{H})$ is closed and bounded; whence compact.

Closedness of the graph follows by the standard compactness argument as the target space is locally compact,
{and conditions (\ref{c1})--(\ref{c3ab}) are invariant with respect to strong convergence.}
\end{proof}

\begin{Remark} \label{cR1}

The mapping assigns to any point $[\vc{w}, \mathbb{H}] \in \mathcal{S}[e]$ a segment $[ \vc{w} + \lambda \vc{s}, \mathbb{H} + \lambda \mathbb{M} ]$,
$\lambda \in [-L,L]$ that has ``maximal'' length and still belongs to the set $\mathcal{S}[e]$. Solutions constructed later by the method of convex integration
``oscillate'' along segments of this type.

\end{Remark}

Let $(\Omega, \mathfrak{F}, \prst)$ be a probability space endowed with a complete $\sigma$-algebra of measurable sets
$\mathfrak{F}$. Suppose now that
\[
[e, \vc{w} , \mathbb{H} ] \ \mbox{is an}\ \left[\mathfrak{F}, \mathfrak{B}[ (0, \infty) \times R^N \times R^{N \times N}_{0, {\rm sym}}] \right]
\mbox{-measurable random variable},
\]
where the symbol $\mathfrak{B}$ denotes the $\sigma$-algebra of Borel sets. Our goal is to show that the composed mapping $\mathcal{F}(e, \vc{w},
\mathbb{H})$, considered now as a (set-valued) random
variable, admits an $\mathfrak{F}$-measurable selection. To this end, we recall the celebrated
Kuratowski and Ryll--Nardzewski theorem, {see
e.g. the survey by Wagner \cite{Wag}. }

\begin{Theorem} \label{TRN}
Let $(X,\mathcal A,\mu)$ be a measure space
with a (complete) $\sigma$-algebra of measurable sets $\mathcal{A}$.
Let
\[
\mathcal{H}: X \to 2^Y
\]
be a set valued mapping, where $Y$ is a Polish space with the $\sigma$-algebra of Borel sets $\mathfrak{B}$.
Suppose that for all $x\in X$
\[
\mathcal{H}(x) \ \mbox{is a nonempty and closed subset of} \ Y,
\]
and that $\mathcal{F}$ is weakly measurable, meaning
\[
\left\{ x \ \Big| \ \mathcal{H}(x) \cap B \ne \emptyset \right\} \in \mathfrak{F}
\]
for any open set $B \subset Y$.

Then $\mathcal{H}$ admits an $\mathcal{A}$--$\mathfrak{B}$ measurable selection, meaning a single valued $\mathcal{A}$--$\mathfrak{B}$
measurable mapping
$H: X \to Y$ such that
\[
H(x) \in \mathcal{H}(x), \ x \in X.
\]

\end{Theorem}

As both spaces $(0, \infty) \times R^N \times R^{N \times N}_{0, {\rm sym}}$ and $R^N \times R^N$ are finite dimensional, compactness of the
range of $\mathcal{F}$ and closedness of its graph implies that $\mathcal{F}$ is \emph{upper semi--continuous}, specifically,
\[
\left\{ [ e, \vc{w}, \mathbb{H} ]  \ \Big| \
\mathcal{F} (e, \vc{w}, \mathbb{H} ) \cap D \ne \emptyset \right\}
\ \mbox{is closed whenever}\ D \ \mbox{is closed in}\ R^N \times R^N,
\]
see Wagner \cite{Wag}.

As pre-images of closed sets are measurable, we get (strong) measurability of $\mathcal{F}$, specifically,
\[
\left\{ \omega \in \Omega \ \Big| \ \mathcal{F}(e, \vc{w}, \mathbb{H} ) \cap D \ne \emptyset \right\}
\]
is measurable for any closed set $D$ in $R^N \times R^{N \times N}_{0, {\rm sym}}$.
Strong measurability implies weak measurability of $\mathcal{F}$, namely,
\[
\left\{ \omega \in \Omega \ \Big| \ \mathcal{F}(e, \vc{w}, \mathbb{H} ) \cap G \ne \emptyset \right\}
\]
is measurable for any open set $G$ in $R^N \times R^N$.
Applying Theorem \ref{TRN} we may infer that the mapping $\mathcal{F}$ admits an
$[\mathfrak{F}; \mathfrak{B}[R^N \times R^N]]$-measurable selection. In particular, there exists an
$[\mathfrak{F}; \mathfrak{B}[ R^N \times R^N]]$-measurable mapping
\[
F: \Omega \to R^N \times R^N
\]
such that it holds $\prst$-a.s.:
\begin{equation} \label{c8}
\mbox{if}\ [\vc{w}(\omega), \mathbb{H}(\omega)] \in \mathcal{S}[e(\omega)], \ \mbox{then}\
F(\omega) = [\vc{a}, \vc{b}],
\ \mbox{
where}\ [\vc{a}, \vc{b} ]\ \mbox{{satisfy (\ref{c1})--(\ref{c3ab}).}}
\end{equation}

\subsection{Analytic setting}

Following \cite{DelSze3} we introduce a mapping
\[
\xi = [\xi_0, \xi_1, \dots, \xi_n ] \mapsto \mathbb{A}(\xi) \in R^{(N + 1) \times (N+1)}_{0 , {\rm sym}},
\]
\begin{equation} \label{d1}
\mathbb{A}_{\vc{a}, \vc{b}} (\xi) = \frac{1}{2} \left( (\mathbb{R} \cdot \xi) \otimes (\mathbb{Q} (\xi) \cdot \xi ) + (\mathbb{Q} (\xi) \cdot \xi )
\otimes (\mathbb{R} \cdot \xi)
\right)
\end{equation}
where
\[
\mathbb{Q} = \xi \otimes e_{0} - e_0 \otimes \xi , \ \mathbb{R} = \left( [0, \vc{a}] \otimes [0, \vc{b}] \right) -  \left(
[0, \vc{b}] \otimes [0, \vc{a}] \right),
\]
and
\[
e_0 = [1,0,\dots,0],\
\vc{a},  \vc{b} \in R^N,\
\frac{1}{2} |\vc{a}|^2 = \frac{1}{2} |\vc{b}|^2 = e > 0,\ \vc{a} \ne \pm \vc{b}.
\]
$\mathbb{A}_{\bf{a}, \bf{b}}$ can be seen as a Fourier symbol of a pseudo--differential operator, where $\xi = (\xi_0, \xi_1, \dots, \xi_N)$
corresponds to $\partial = [\partial_t, \partial_{x_1}, \dots, \partial_{x_N}]$.

The following was shown in \cite{DelSze3}:
\begin{itemize}
\item if $\phi \in \DC(R \times R^N)$, then
\begin{equation} \label{d2}
\mathbb{A}_{\vc{a}, \vc{b}} (\partial) [\phi] = \left[ \begin{array}{cc} 0 & \vc{w} \\ \vc{w} & \mathbb{H} \end{array}  \right]
\ \mbox{satisfies}\ \partial_t \vc{w} + \Div \mathbb{H} = 0, \ \Div \vc{w} = 0;
\end{equation}

\item
for
\begin{equation} \label{d3}
\eta_{\vc{a}, \vc{b}} = - \frac{1}{\left( |\vc{a}| |\vc{b}| + \vc{a} \cdot \vc{b} \right)^{2/3} } \Big[ [0,\vc{a}] + [0,\vc{b}] -
\left( |\vc{a}| |\vc{b}| + \vc{a} \cdot \vc{b} \right) e_0 \Big],\ \psi \in C^\infty(R),
\end{equation}
we have
\begin{equation} \label{d4}
\mathbb{A}_{\vc{a}, \vc{b}}(\partial) [\psi ( [t, \vc{x}] \cdot \eta_{\vc{a}, \vc{b}} )]  =
\psi''' ( [t, \vc{x}] \cdot \eta_{\vc{a}, \vc{b}} ) \left[ \begin{array}{cc} 0 & \vc{a} - \vc{b} \\ \vc{a} - \vc{b} & \vc{a} \otimes
\vc{a} - \vc{b} \otimes \vc{b} \end{array}  \right].
\end{equation}

\end{itemize}

\subsection{Stochastic version of oscillatory lemma}

Let $Q = \big\{ (t,x) \ \big| \ t \in (0,1),\ x \in (0,1)^N \big\}$. Let $( \Omega, \mathfrak{F}, \prst )$ be a probability space
with  a complete $\sigma$-algebra of measurable sets $\mathfrak{F}$. The following is the main result of the present section.

\begin{Lemma} \label{OL1}

Let $\omega \mapsto [e, \vc{w}, \mathbb{H}]$ be a $[\mathfrak{F}; \mathfrak{B}[ (0, \infty) \times R^N \times R^{N \times N}_{0, {\rm sym}} ] ]$ measurable mapping such that
\begin{equation} \label{aHYP1}
[\vc{w}, \mathbb{H}] \in \mathcal{S}[e] \ \prst\mbox{-a.s.}
\end{equation}

Then there exists a sequence $\vc{w}_n \in \DC(Q; R^N)$ $\pas$ and $\mathbb{V}_n \in \DC (Q; R^{N \times N}_{0, {\rm sym}} )$ $\pas$, $n\in\mathbb{N}$, enjoying the following
properties:
\begin{enumerate}[i)]
\item
$t \mapsto [\vc{w}_n, \mathbb{V}_n]$ is a stochastic process, meaning
\begin{equation} \label{d5}
\begin{split}
[ \vc{w}_n (t, \cdot); \mathbb{V}_n (t, \cdot) ] &\in C([0,1]^N; R^N \times R^{N \times N}_{0, {\rm sym}})\ \prst\mbox{-a.s.} \\ &\mbox{is}\ \Big[ \mathfrak{F}; \mathfrak{B}[C([0,1]^N; R^N \times R^{N \times N}_{0, {\rm sym}} ) \Big]\mbox{-measurable}
\ \mbox{for any}\ t \in [0,1];
\end{split}
\end{equation}
\item  in $Q$ we have $\prst$-a.s.
\begin{equation} \label{d6}
\partial_t \vc{w}_n + \Div \mathbb{V}_n = 0, \ \Div \vc{w}_n = 0;
\end{equation}

\item as $n \to \infty$ we have $\prst$-a.s.
\begin{equation} \label{d7}
\vc{w}_n \to 0 \ \mbox{in}\ C_{\rm weak}([0,1]; L^2([0,1]^N; R^N)) ;
\end{equation}

\item in $Q$ we have $\prst$-a.s.
\begin{equation} \label{d8}
[\vc{w} + \vc{w}_n, \ \mathbb{H} + \mathbb{V}_n ] \in \mathcal{S}[e];
\end{equation}

\item the following holds $\prst$-a.s.

\begin{equation} \label{d9}
\liminf_{n \to \infty} \frac{1}{|Q|} \int_Q |\vc{w}_n |^2 \ \dxdt \geq \frac{c(N)}{{e}}
\left( e - \frac{1}{2} |\vc{w}|^2 \right)^2 .
\end{equation}

\end{enumerate}

If, in addition to (\ref{aHYP1}), $e \leq \Ov{e}_M$ $\prst$-a.s.,
and
\begin{equation} \label{aHYP2}
[\vc{w}, \mathbb{H}] \in \mathcal{S}[e - \delta] \ \mbox{for some deterministic constant}\ \delta > 0,
\end{equation}
then each $\vc{w}_n$, $\mathbb{V}_n$ has compact range in
$C(\Ov{Q}; R^N)$, $C(\Ov{Q}; R^{N \times N}_{0, {\rm sym}})$, and
\begin{equation} \label{d8a}
[\vc{w} + \vc{w}_n, \ \mathbb{H} + \mathbb{V}_n ] \in \mathcal{S}[e - \delta_n] \ \prst\mbox{-a.s.};
\end{equation}
for some deterministic constants $\delta_n > 0$.

\end{Lemma}

\begin{Remark} \label{aR5}
Hypothesis (\ref{aHYP2}) is equivalent to saying that
\[
{\rm ess}\inf_{\Omega} \left\{ e - \frac{N}{2} \lambda_{\rm max} \left[ \vc{w} \otimes \vc{w} - \mathbb{H} \right] \right\} > 0.
\]
Note that if this is the case, we have $e \geq \delta > 0$; whence $e$ is a random variable with a compact range in $(0, \infty)$.

\end{Remark}

\begin{proof}

The proof is given through several steps.

\medskip

{\bf Step 1:}

Given $[\vc{w}, \mathbb{H}]$ and $e$ we identify the measurable selection of vectors $[\vc{a}, \vc{b}]$ satisfying (\ref{c8}).

\medskip

{\bf Step 2:}

For each $[\vc{a}, \vc{b}]$ we construct the operator $\mathbb{A}_{\vc{a}, \vc{b}}$ and the vector $\eta_{\vc{a}, \vc{b}}$ enjoying
(\ref{d2}--\ref{d4}).

\medskip

{\bf Step 3:}

We consider a deterministic function $\varphi \in \DC(Q)$ such that
\[
0 \leq \varphi \leq 1, \ \varphi(t,x) = 1 \ \mbox{whenever} \ - \frac{1}{2} \leq t \leq \frac{1}{2}, \ |\vc{x}| \leq \frac{1}{2}.
\]

\medskip

{\bf Step 4:}

We identify the functions $\vc{w}_n$, $\mathbb {V}_n$ from the relation
\[
\mathbb{A}_{\vc{a}, \vc{b}}(\partial ) \left[
\varphi \frac{L}{n^3} \cos \left( n [t,x] \cdot \eta_{\vc{a}, \vc{b}} \right) \right] =
\left[ \begin{array}{cc} 0 & \vc{w}_n \\ \vc{w}_n & \mathbb{V}_n \end{array}  \right].
\]

In accordance with our construction of the points $[\vc{a}, \vc{b}]$, the operator $\mathbb{A}_{\vc{a}, \vc{b}}$, and the vector
$\eta_{\vc{a}, \vc{b}}$, it is easy to check the $\vc{w}_n$, $\mathbb{V}_n$ enjoy the required measurability properties
(\ref{d5}). Moreover, by virtue of (\ref{d2}), equations (\ref{d6}) are satisfied.

\medskip

{\bf Step 5:}

As $\mathbb{A}$ is a homogeneous differential operator of third order, we get, in agreement with (\ref{d4}),
\begin{equation} \label{pompom}
\mathbb{A}_{\vc{a}, \vc{b}}(\partial ) \left[
\varphi \frac{L}{n^3} \cos \left( n [t,x] \cdot \eta_{\vc{a}, \vc{b}} \right) \right]
= \varphi \sin \left( n [t,x] \cdot \eta_{\vc{a}, \vc{b}} \right)
L \left[ \begin{array}{cc} 0 &  (\vc{a} - \vc{b})  \\  (\vc{a}  -\vc{b})  & \vc{a} \otimes \vc{a} - \vc{b} \otimes \vc{b} \end{array}  \right]
+ \frac{1}{n} R_n
\end{equation}
with $|R_n|$ uniformly bounded for $n \to \infty$. As (\ref{c2}), (\ref{c3}) holds, we deduce the remaining properties (\ref{d7}--\ref{d9})
provided $n$ is chosen large enough. {Note that we have}
 \begin{align*}|\varphi \sin \left( n [t,x] \cdot \eta_{\vc{a}, \vc{b}} \right)|\leq 1
\end{align*}
{and}
\begin{align*}
\int_Q |\vc{w}_n |^2 \ \dxdt &\geq \frac{c}{{e}}
\left( e - \frac{1}{2} |\vc{w}|^2 \right)^2 \int_Q \varphi^2\sin^2(n[t,x]\cdot\eta_{\vc{a},\vc{b}})\dx\dt-\frac{c|R_n|^2}{n^2}\\
&=\frac{c}{{e}}
\left( e - \frac{1}{2} |\vc{w}|^2 \right)^2 \frac{|Q|}{2}-\frac{c|R_n|^2}{n^2}
\end{align*}
using \cite[Lemma 7]{DelSze3} in the last step.
Strictly speaking $|R_n|$ is a random variable so we need $n \geq n_0(\omega)$, where
the latter is $\mathfrak{F}$-measurable. Setting $[\vc{w}_n, \mathbb{V}_n] = [0,0]$ whenever $n \leq n_0$ yields the desired inclusion
(\ref{d8}).

{\bf Step 6:}

Finally, if $e \leq \Ov{e}_M$ for some deterministic constants, then $\vc{w}$, $\mathbb{H}$ have compact range in $R^N$, $R^{N \times N}_{0, {\rm sym}}$, respectively. In addition, hypothesis (\ref{aHYP2}) implies that
\[
[\vc{w}, \mathbb{H}] \in \mathcal{S}[e - \ep] \ \mbox{for any}\ 0 \leq \ep < \delta.
\]
Thus the above construction can be therefore repeated with $e$ replaced by $e - \ep$, $\ep > 0$.
{Moreover, in view of (\ref{c3ab}), the remainder $R_n$ specified in {\bf Step 5} above is now bounded uniformly by a deterministic constant
depending only on $\ep$.} Since
\[
\Ov{ \mathcal{S}[e - \delta] } \subset \mathcal{S}[e - \varepsilon] \subset \Ov{\mathcal{S}[e - \varepsilon] } \subset \mathcal{S}[e],
\]
compactness of the range of $\vc{w}_n$, $\mathbb{V}_n$
follows from their construction {and (\ref{c3ab}). Notably relations \eqref{c1} and (\ref{c3ab}) yield deterministic (in terms of $\ep$) upper and lower
bounds on the norm of the vector $\eta_{\vc{a}, \vc{b}}$ used in the construction of $\vc{w}_n$, $\mathbb{V}_n$.}

As $\ep > 0$ can be taken arbitrarily small, the desired conclusion follows.

\end{proof}

\subsubsection{Extension by scaling}

Let
\[
Q = (T_1, T_2) \times (a_1, b_1) \times \dots \times (a_N, b_N).
\]
Following \cite[Section 4.2]{DoFeMa}, we may use scaling in $t$ and $x$ and additivity of the integral to show the following extension of Lemma \ref{OL1}.

\begin{Lemma} \label{OL2}

Let $\omega \mapsto [e,r, \vc{w}, \mathbb{H}]$ be a $[\mathfrak{F}; \mathfrak{B}[ (0, \infty)^2, R^N, R^{N \times N}_{0, {\rm sym}} ] ]$-measurable mapping such that
\[
\left[ \frac{\vc{w}}{\sqrt{r}}, \mathbb{H} \right] \in \mathcal{S}[e] \ \prst\mbox{-a.s.}
\]

Then there exists a sequence $\vc{w}_n \in \DC(Q; R^N)$ $\pas$ and $\mathbb{V}_n \in \DC (Q; R^{N \times N}_{0, {\rm sym}} )$ $\pas$, $n\in\mathbb{N}$, enjoying the following
properties:

\begin{enumerate}[i)]
\item
$t \mapsto [\vc{w}_n, \mathbb{V}_n]$ is a stochastic process, meaning
\begin{equation} \label{d10}
\begin{split}
[ \vc{w}_n (t, \cdot); \mathbb{V}_n (t, \cdot) ] &\in C(\Pi_{i=1}^N [a_i, b_i]; R^N \times R^{N \times N}_{0, {\rm sym}})\ \prst\mbox{-a.s.} \\ &\mbox{is}\ \Big[ \mathfrak{F}; \mathfrak{B}[C(\Pi_{i=1}^N [a_i, b_i]; R^N \times R^{N \times N}_{0, {\rm sym}} ) \Big]\mbox{-measurable}
\ \mbox{for any}\ t \in [T_1,T_2];
\end{split}
\end{equation}
\item in $Q$ we have $\prst$-a.s.
\begin{equation} \label{d11}
\partial_t \vc{w}_n + \Div \mathbb{V}_n = 0, \ \Div \vc{w}_n = 0;
\end{equation}

\item as $n\rightarrow\infty$ we have $\prst$-a.s.
\begin{equation} \label{d12}
\vc{w}_n \to 0 \ \mbox{in}\ C_{\rm weak} ([T_1,T_2]; L^2(R^N; R^N));
\end{equation}

\item in $Q$ we have $\prst$-a.s.
\begin{equation} \label{d13}
\left[ \frac{\vc{w} + \vc{w}_n}{\sqrt{r}}, \ \mathbb{H} + \mathbb{V}_n \right] \in \mathcal{S}[e] ;
\end{equation}

\item the following holds $\prst$-a.s.

\begin{equation} \label{d14}
\liminf_{n \to \infty} \frac{1}{|Q|} \int_Q \frac{|\vc{w}_n |^2}{r} \ \dxdt \geq \frac{c(N)}{{e}}
\left( e - \frac{1}{2} \frac{|\vc{w}|^2}{r} \right)^2.
\end{equation}

\end{enumerate}

If, in addition,
\begin{equation} \label{cRc}
0 < \underline{r}_M \leq r \leq \Ov{r}_M, \ 0 < \underline{e}_M \leq e \leq \Ov{e}_M \ \prst\mbox{-a.s.}
\end{equation}
for some deterministic constants $\underline{r}_M$, $\Ov{r}_M$, $\underline{e}_M$, $\Ov{e}_M$,
and
\[
\left[ \frac{\vc{w}}{\sqrt{r}}, \mathbb{H} \right] \in \mathcal{S}[e - \delta] \ \prst\mbox{-a.s. for some deterministic}\ \delta > 0,
\]
then each $\vc{w}_n$, $\mathbb{V}_n$ has compact range in
$C(\Ov{Q}; R^N)$, $C(\Ov{Q}; R^{N \times N}_{0, {\rm sym}})$, respectively, and
\begin{equation} \label{d13a}
\left[ \frac{\vc{w} + \vc{w}_n}{\sqrt{r}}, \ \mathbb{H} + \mathbb{V}_n \right] \in \mathcal{S}[e - \delta_n] \ \prst\mbox{-a.s.
for some deterministic}\ \delta_n > 0.
\end{equation}

\end{Lemma}

\begin{Remark} \label{AdR4}

Condition (\ref{cRc}) can be equivalently formulated saying that the random variable $[r, e]$ has compact range in $(0, \infty)^2$.

\end{Remark}

\subsubsection{Extension to piecewise constant coefficients}

Consider now a complete right-continuous filtration $(\mathfrak{F}_t)_{t \geq 0}$ of measurable sets in
$\Omega$ and fix
$Q = (0,T) \times (0,1)^N$. We write $[0,1]^N = \cup_{i \in I} \overline K_i$, where $K_i$ are disjoint open cubes of the edge length $\frac{1}{m}$ for some $m\in\mathbb N$.
The random variables $e$, $r$, $\vc{w}$, and $\mathbb{H}$ will be now $\prst$-a.s. functions of the time $t$ and the spatial variable $x$ that are piecewise constant. More specifically, they shall $\prst$-a.s. belong to the class of functions satisfying
\begin{equation} \label{p1}
F(t,x) = F_{j,i} \ \mbox{whenever}\ t \in \left[ \frac{jT}{m}; \frac{(j+1)T}{m} \right),\ x \in K_i, \ 0 \leq j \leq m-1,\ i \in I.
\end{equation}
These functions are piecewise constant on the rectangular grid given by
\[
\left[ \frac{jT}{m}, \frac{(j+1)T}{m} \right) \times K_i,\
0 \leq j \leq m -1, \ i \in I.
\]
In addition, we suppose that $[e,r,\vc{w},\mathbb H]$ is $(\mathfrak{F}_t)$-adapted, meaning that
\[
[e,r,\vc{w},\mathbb H](t,\cdot) \ \mbox{is}\ \mathfrak{F}_{\frac{jT}{m}} \ \mbox{-measurable whenever}\
t \in \left[ \frac{jT}{m}; \frac{(j+1)T}{m} \right).
\]

Keeping in mind that the oscillatory increments $[\vc{w}_n, \mathbb{V}_n]$ constructed in Lemma \ref{OL2}
are compactly supported in each cube and hence globally smooth,
we get the following result when applying  Lemma \ref{OL2} with $\mathfrak F$
replaced by $\mathfrak{F}_{\frac{jT}{m}}$.
{Note that $\vc{w}_n,\mathbb V_n$ are even $\mathfrak{F}_{\frac{jT}{m}}$ adapted}
\begin{Lemma} \label{OL3}

Let $\bas$ be a probability space with a complete right continuous filtration $(\mathfrak{F}_t)_{t \geq 0}$.
Let $[e,r, \vc{w}, \mathbb{H}]$ be an $(\mathfrak{F}_t)$-adapted stochastic process which is $\prst$-a.s. piecewise constant and belongs to the class (\ref{p1}). Suppose further that $r > 0$, $e > 0$
$\prst$-a.s.
and
\begin{equation} \label{TQ1}
\left[ \frac{\vc{w}}{\sqrt{r}}, \mathbb{H} \right] \in \mathcal{S}[e] \ \mbox{for all}\ (t,x) \in \Ov{Q} \ \prst\mbox{-a.s.}
\end{equation}

Then there exists a sequence $\vc{w}_n \in \DC(Q; R^N)$ $\pas$ and $\mathbb{V}_n \in \DC (Q; R^{N \times N}_{0, {\rm sym}} )$ $\pas$, $n\in\mathbb{N}$, enjoying the following
properties:

\begin{enumerate}[i)]
\item
the process $[\vc{w}_n,\tn{V}_n]$ is $(\mathfrak{F}_t)$-adapted such that
\[
[\vc{w}_n,\tn{V}_n] \in C(\Ov{Q}; R^N \times R^{N \times N}_{0, {\rm sym}})\ \prst\mbox{-a.s.} \ \mbox{with compact range};
\]
\item in $Q$ we have $\prst$-a.s.
\[
\partial_t \vc{w}_n + \Div \mathbb{V}_n = 0, \ \Div \vc{w}_n = 0;
\]

\item as $n\rightarrow\infty$ we have $\prst$-a.s.
\[
\vc{w}_n \to 0 \ \mbox{in}\ C_{\rm weak}([0,T]; L^2(\TN; R^N)) ;
\]

\item in $Q$ we have $\prst$-a.s.
\begin{equation} \label{TQ2}
\left[ \frac{\vc{w} + \vc{w}_n}{\sqrt{r}}, \ \mathbb{H} + \mathbb{V}_n \right] \in \mathcal{S}[e];
\end{equation}

\item the following holds $\prst$-a.s.

\begin{equation} \label{TQ3}
\liminf_{n \to \infty}  \int_Q \frac{|\vc{w}_n |^2}{r} \ \dxdt \geq \frac{c(N)}{\sup_Q e} \int_Q
\left( e - \frac{1}{2} \frac{|\vc{w}|^2}{r} \right)^2 \ \dxdt.
\end{equation}

\end{enumerate}

If, in addition,
\[
0 < \underline{r}_M \leq r \leq \Ov{r}_M, \ 0 < \underline{e}_M \leq e \leq \Ov{e}_M \ \prst\mbox{-a.s.}
\]
for some deterministic constants $\underline{r}_M$, $\Ov{r}_M$, $\underline{e}_M$, $\Ov{e}_M$,
and
\[
\left[ \frac{\vc{w}}{r}, \mathbb{H} \right] \in \mathcal{S}[e - \delta] \ \prst\mbox{-a.s. for some deterministic}\ \delta > 0,
\]
then each $\vc{w}_n$, $\mathbb{V}_n$ has compact range in
$C(\Ov{Q}; R^N)$, $C(\Ov{Q}; R^{N \times N}_{0, {\rm sym}})$, respectively, and
\[
\left[ \frac{\vc{w} + \vc{w}_n}{\sqrt{r}}, \ \mathbb{H} + \mathbb{V}_n \right] \in \mathcal{S}[e - \delta_n] \ \prst\mbox{-a.s.
for some deterministic}\ \delta_n > 0.
\]

\end{Lemma}

\subsubsection{Extension to continuous coefficients}

Using the result on the piecewise constant coefficients we may use the approximation procedure from \cite[Section 4.3]{DoFeMa}
to extend the oscillatory lemma to the class of continuous processes $[e,r,\vc{w}, \mathbb{H}]$. The obvious idea is to replace
$[e,r,\vc{w}, \mathbb{H}]$ by piecewise constant approximations and apply Lemma \ref{OL3}. More specifically,
for $e \in C([0,T] \times \TN; (0, \infty) )$ $\mathbb P$-a.s., $(\mathfrak{F}_t)$-adapted, $e > 0$ $\prst$-a.s., we define a piecewise constant approximation
\begin{equation} \label{AP1}
\begin{split}
e_m(t,x) &= \sup_{ y \in K_i } e \left( \frac{jT}{m}, y \right)
\\ &\mbox{for}\ t \in \left[ \frac{jT}{m}; \frac{(j+1)T}{m} \right),\
x \in K_i,\ 0 \leq j \leq m-1,\  i \in I,
\end{split}
\end{equation}
and, similarly, for $F\in\{r,\vc{w}, \mathbb{H}\}$,
\begin{equation} \label{AP2}
\begin{split}
F_m(t,x) &= F \left( \frac{jT}{m}, y \right) \ \mbox{for some}\  y \in K_i,
\\ &\mbox{for}\ t \in \left[ \frac{jT}{m}; \frac{(j+1)T}{m} \right),\
x \in K_i,\ 0 \leq j \leq m-1, \ i \in I.
\end{split}
\end{equation}
It is easy to check that these approximations satisfy the hypotheses of Lemma \ref{OL3}.

Now, since $\mathcal{S}[e]$ is an open set, it is possible, similarly to \cite[Section 4.3]{DoFeMa} to replace
$F_m$ by $F$ as long as the approximation is uniform. Specifically, for any $\delta > 0$, there is $m = m(\delta)$ such that
\begin{equation} \label{AP3}
| F_m(t,x) - F(t,x) | < \delta \ \mbox{for all}\ (t,x) \in \Ov{Q}
\end{equation}
$\mathbb P$-a.s.
For (\ref{AP3}) to hold, it is necessary (and sufficient) that all random variables $F = e,r, \vc{w}, \mathbb{H}$ have compact range in the space of
continuous functions on $\Ov{Q}$. Repeating the arguments of \cite[Section 4.3]{DoFeMa} we show the final form of the oscillatory lemma.

\begin{Lemma} \label{OL4}

Let $[\Omega, \mathfrak{F}, \mathfrak{F}_t, \prst]$ be a probability space with a complete right continuous filtration $(\mathfrak{F}_t)_{t \geq 0}$.
Let $[e,r, \vc{w}, \mathbb{H}]$ be an $(\mathfrak{F}_t)$-adapted stochastic process such that
\[
[e,r, \vc{w}, \mathbb{H}] \in C \Big(\Ov{Q}; (0, \infty)^2 \times R^N \times R^{N \times n}_{0, {\rm sym}} \Big) \ \prst\mbox{-a.s.}
\]
 with compact range and such that
\begin{equation} \label{hyphyp}
\left[ \frac{\vc{w}}{\sqrt{r}}, \mathbb{H} \right] \in \mathcal{S}[e - \delta] \ \mbox{for all}\ (t,x) \in \Ov{Q} \ \prst\mbox{-a.s.}
\end{equation}
for some deterministic constant $\delta > 0$.

Then there exists a sequence $\vc{w}_n \in \DC(Q; R^N)$ $\pas$ and $\mathbb{V}_n \in \DC (Q; R^{N \times N}_{0, {\rm sym}} )$ $\pas$, $n\in\mathbb N$, enjoying the following
properties:

\begin{enumerate}[i)]
\item
the process $[\vc{w}_n,\tn{V}_n]$ is $(\mathfrak{F}_t)$-adapted such that
\[
[\vc{w}_n,\tn{V}_n] \in C(\Ov{Q}; R^N \times R^{N \times N}_{0, {\rm sym}})\ \prst\mbox{-a.s.} \ \mbox{with compact range};
\]
\item in $Q$ we have $\prst$-a.s.
\[
\partial_t \vc{w}_n + \Div \mathbb{V}_n = 0, \ \Div \vc{w}_n = 0
;
\]

\item as $n\rightarrow\infty$ we have $\prst$-a.s.
\[
\vc{w}_n \to 0 \ \mbox{in}\ C_{\rm weak}([0,T]; L^2(\TN; R^N));
\]

\item in $Q$ we have $\prst$-a.s.
\[
\left[ \frac{\vc{w} + \vc{w}_n}{\sqrt{r}}, \ \mathbb{H} + \mathbb{V}_n \right] \in \mathcal{S}[e - \delta_n]
\]
for some deterministic $\delta_n > 0$;

\item the following holds $\prst$-a.s.

\begin{equation} \label{oscil}
\liminf_{n \to \infty}  \int_Q \frac{|\vc{w}_n |^2}{r} \ \dxdt \geq \frac{c(N)}{\sup_Q e} \int_Q
\left( e - \frac{1}{2} \frac{|\vc{w}|^2}{r} \right)^2 \ \dxdt.
\end{equation}

\end{enumerate}

\end{Lemma}

\begin{Remark} \label{OLR1}

Observe that the assumption for a random variable $[e,r, \vc{w}, \mathbb{H}]$ to be of compact range in $C \Big(\Ov{Q}; (0, \infty)^2 \times R^N \times R^{N \times n}_{0, {\rm sym}} \Big)$ includes
\[
0 < \underline{r}_M \leq r \leq \Ov{r}_M, \ e \leq e_M \ \prst\mbox{-a.s.}
\]
for some deterministic constants $\underline{r}_M$, $\Ov{r}_M$, $e_M$ {as well as
a positive lower bound for $e$ already guaranteed by (\ref{hyphyp}).}

\end{Remark}

\begin{Remark} \label{OLR2}

The fact that the continuous processes considered in Lemma \ref{OL4} must have compact range
is definitely restrictive but possibly unavoidable. This is also the main reason why our result holds up to a stopping time albeit arbitrarily large
with ``high'' probability. Otherwise, the size of the grid used to construct the approximations $F_m$ would have to be a random variable. The oscillatory increments
$\vc{w}_n$ would be then constructed on a grid determined by random points $0 < t_1 < t_2 < \dots <t_m $ related to stopping times associated to certain norms of the random processes. Here, the lengths of the interval
$[t_m, t_{m+1}]$ would have to be $t_m$ predictable which seems impossible.

\end{Remark}

\section{Infinitely many solutions}
\label{Inf}

We are ready to show Theorem \ref{MT1} or, equivalently, the abstract ``Euler'' problem (\ref{AE5}), (\ref{AE9}), respectively. We focus on problem (\ref{AE5}),
where the tensor $\mathbb{M}$ is constant, referring to \cite{Fei2016} for how to accommodate the dependence $\mathbb{M} = \mathbb{M}[\vc{v}]$.

\subsection{Subsolutions}
\label{sec:subsol}
We introduce the set
\begin{align*}
\mathcal X(R^N):=\left\{\vc{v}:\Omega\times Q\rightarrow R^N \ \Big| \ \vc{v} \in C([0,T] \times \TN; R^N)\,\,\mathbb P\text{-a.s. with compact range}\right\}
\end{align*}
and analogously $\mathcal X(R^{N \times N}_{0,{\rm sym}})$.
Following \cite{DelSze3}, we introduce the
set of \emph{subsolutions}.
Let the functions $\vc{v}_0$, $\vc{h}$, $r$ and $\mathbb{M}$ satisfy (\ref{AE6}), and $e = e(t)$ is a real valued $(\mathfrak F_t)$-adapted process specified below.
In particular, the process $\left[\vc{h},r, \mathbb{M}] \in C([0,T] \times \TN; R^N \times (0, \infty) \times R^{N \times N}_{0, {\rm sym}} \right]$
is $(\mathfrak{F}_t)_{t \geq 0}$ adapted and with compact range.
\begin{align}
X_0 &= \left\{ \vc{v}\in \mathcal X(R^N) \ \Big| \ \mbox{$\vc{v}$ is $(\mathfrak{F}_t)$-adapted with} \
\vc{v}(0, \cdot) = \vc{v}_0 ,\ \mbox{there is }\tn{F}\in \mathcal X(R^{N \times N}_{0,{\rm sym}})\,\mbox{$(\mathfrak{F}_t)$-adapted}   \right.\nonumber
\\
&\qquad\ \partial_t \vc{v} + \Div \tn{F} = 0,
 \ \Div \vc{v} = 0 \ \mbox{in}\ \mathcal{D}'((0,T) \times \TN; R^N)\,\,\mathbb P\text{-a.s.},
\nonumber\\
&
 \left.\qquad\frac{N}{2}\lambda_{\rm max} \left[ \frac{ (\vc{v} + \vc{h} ) \otimes (\vc{v} + \vc{h}) }{r} - \tn{F} + \tn{M} \right]
< e - \delta \ \ \forall \ 0 \leq t \leq T, \ x \in \TN,\,\,\mathbb P\text{-a.s.} \right.\nonumber\\& \left.
\qquad\mbox{for some deterministic} \ \delta > 0\vphantom{\frac{1}{2}}  \right\}. \label{E6}
\end{align}
Here the functions $\vc{v}_0$, $\vc{h}$, $r$ and $\mathbb{M}$ satisfy (\ref{AE6}), and $e = e(t)$ is a real valued process specified below.
In particular, the process $\left[\vc{h},r, \mathbb{M}] \in C([0,T] \times \TN; R^N \times (0, \infty) \times R^{N \times N}_{0, {\rm sym}} \right]$ $\pas$
is $(\mathfrak{F}_t)_{t \geq 0}$ adapted and with compact range.

\begin{Remark} \label{RemFi}

{
The deterministic constant $\delta > 0$ may vary from one subsolution to another. The exact meaning of the condition}
\[
\frac{N}{2}\lambda_{\rm max} \left[ \frac{ (\vc{v} + \vc{h} ) \otimes (\vc{v} + \vc{h}) }{r} - \tn{F} + \tn{M} \right]
< e - \delta
\]
{is}
\[
{\rm ess} \sup_{\Omega} \sup_{t \in [0,T]; x \in \TN}\left(
\frac{N}{2}\lambda_{\rm max} \left[ \frac{ (\vc{v} + \vc{h} ) \otimes (\vc{v} + \vc{h}) }{r} - \tn{F} + \tn{M} \right] -e\right)
< 0.
\]

\end{Remark}

\subsection{Existence of a subsolution}

Next we claim that $e$ can be fixed in such a way that the set of subsolutions is non-empty. To this end, consider
\[
\vc{v} = \vc{v}_0, \ \mathbb{F} = 0.
\]
This is obviously a subsolution provided $e$ is taken in such a way that
\[
\frac{N}{2}\lambda_{\rm max} \left[ \frac{ (\vc{v}_0 + \vc{h} ) \otimes (\vc{v}_0 + \vc{h}) }{r}  + \tn{M} \right]
< e - \delta.
\]
In view of (\ref{AE6}) this is possible, where $e = e_M$ can be taken a sufficiently large \emph{deterministic} constant.

\subsection{Topology on the set of subsolutions}

The processes $\vc{v}$ belonging to $X_0$ are uniformly \emph{deterministically} bounded in $L^\infty((0,T) \times \TN)$, specifically,
$\vc{v}(t) \in B_\infty$ for any $t \in [0,T]$ $\prst$-a.s., where $B_\infty$ is a ball in $L^\infty(\TN)$ with a deterministic radius. Consequently, we may consider
the metric $d$ associated to the weak $L^2(\TN)$-topology on $B_\infty$, together with
\[
D[\vc{v}, \vc{w}] = \expe{ \sup_{t \in [0,T]} d[\vc{v}(t); \vc{w}(t)] }.
\]

Let $X$ be the completion of $X_0$ with respect to the metric $D$. Thus $X$ is a complete metric space with infinite cardinality. Note that
any element of $X$ is $(\mathfrak F_t)$-adapted as the limit of measurable functions is measurable.

\subsection{Convex functional}
\label{sub:var}

Similarly to \cite{DelSze3}, we introduce the functional
\[
I[\vc{v}] = \expe{ \int_0^T \intTN{ \left[ \frac{1}{2} \frac{|\vc{v} + \vc{h}|^2}{r} - e \right] } \dt }.
\]
Exactly as in \cite{DelSze3}, it can be shown that:
\begin{itemize}
\item $I$ is lower semi-continuous on the space $X$;
\item $I[\vc{v}] \leq 0$ for any $\vc{v} \in X$;
\item if $I[\vc{v}] = 0$ then
\begin{equation} \label{CI1}
e = \frac{1}{2} \frac{|\vc{v} + \vc{h}|^2}{r} \ \mbox{a.e. in}\ (0,T) \times \TN.
\end{equation}
$\prst$-a.s.
\end{itemize}
We want to prove that each $\vc{v}\in X$ with $I[\vc{v}] = 0$ solves the abstract Euler equation (\ref{AE5}).
Let $\vc{v}\in X$. Then there is $(\vc{v}_m)\subset X_0$ with $\vc{v}_m\rightarrow \vc{v}$
with respect to the metric $D$. By definition of $X_0$ we can find a sequence of $(\mathfrak F_t)$-adapted processes $(\tn{F}_m)$ with $\mathbb{F}_m \in L^\infty(Q,R^{N \times N}_{0, {\rm sym}})$ $\prst$-a.s. such that
\begin{equation} \label{CI2}
\partial_t \vc{v}_m + \Div \mathbb{F}_m = 0 \
\mbox{in}\ \mathcal{D}'((0,T) \times R^N)
\end{equation}
 $\prst$-a.s. and
\[
\frac{N}{2}\lambda_{\rm max} \left[ \frac{ (\vc{v}_m + \vc{h} ) \otimes (\vc{v}_m + \vc{h}) }{r} - \tn{F}_m + \tn{M} \right]
\leq e.
\]
Using \eqref{AlIn2} and the properties of $\tn{M}$ (recall  (\ref{AE6}))
we see that $\tn{F}_m$ is uniformly bounded in $L^\infty(\Omega\times Q,R^{N \times N}_{0, {\rm sym}})$. Hence, after choosing a weakly$^*$-converging subsequence, we obtain
\begin{equation} \label{CI2a}
\partial_t \vc{v} + \Div \mathbb{F} = 0,\ \Div \vc{v} = 0, \ \vc{v}(0, \cdot) = \vc{v}_0 \
\mbox{in}\ \mathcal{D}'((0,T) \times R^N),
\end{equation}
for a certain $(\mathfrak F_t)$-adapted process $\tn{F}$ with $\mathbb{F} \in L^\infty(Q,R^{N \times N}_{0, {\rm sym}})$ $\prst$-a.s.
Due to convexity of the functional
\[
[\vc{v}, \mathbb{F}] \mapsto \frac{N}{2}\lambda_{\rm max} \left[ \frac{ (\vc{v} + \vc{h} ) \otimes (\vc{v} + \vc{h}) }{r} - \tn{F} + \tn{M} \right],
\]
we have
\[
\frac{N}{2}\lambda_{\rm max} \left[ \frac{ (\vc{v} + \vc{h} ) \otimes (\vc{v} + \vc{h}) }{r} - \tn{F} + \tn{M} \right]
\leq e.
\]
 Consequently, by virtue of (\ref{AlIn}), relation (\ref{CI1}) implies
\[
\tn{F} = \tn{M} + \frac{ (\vc{v} + \vc{h} ) \otimes (\vc{v} + \vc{h}) }{r} - \frac{1}{N} \frac{|\vc{v} + \vc{h}|^2}{r}
= \tn{M} + \frac{ (\vc{v} + \vc{h} ) \otimes (\vc{v} + \vc{h}) }{r} - \frac{2}{N} e.
\]
As $e$ is independent of $x$, equation (\ref{CI2}) yields the desired conclusion (\ref{AE5}).

Thus each zero point of $I$ yields a weak solution of the abstract Euler problem (\ref{AE5}). Our next claim is that $I[\vc{v}] = 0$ whenever
$\vc{v}$ is a point of continuity of $I$ on $X$. By means of the Baire category argument, the points of continuity of $I$ are dense in $X$ which completes the proof of existence of infinitely many solutions claimed in Theorem \ref{MT1}. Thus it remains to show that $I$ vanishes at each point of continuity, which is the objective of the last section.

\subsection{Points of continuity of $I$ in $X$}

Let $\vc{v}$ be a point of continuity of $I$ on $X$. Suppose that $I[\vc{v}] < 0$. Consequently, there is a sequence
\[
\vc{v}_m \in X_0, \ D[\vc{v}_m;\vc{v}] \to 0,\ I[\vc{v}_m] \to I[\vc{v}], \ I[\vc{v}_m] < - \ep < 0 \ \mbox{for all}\ m=0,1,\dots
\]
Now, we use the oscillatory lemma (version Lemma \ref{OL4}) with the ansatz $\vc{w} = \vc{v}_m + \vc{h}$, $\mathbb{H} =\mathbb{F}_m- \mathbb{M}$.
Consequently, for each fixed $m$, we find a sequence $\left\{ \vc{w}_{m,n} \right\}_{n=1}^\infty \subset X_0$ such that
\begin{align*}
&\vc{v}_m + \vc{w}_{m,n} \in X_0, \ D[ \vc{v}_m + \vc{w}_{m,n}, \vc{v}_m] \to 0 \ \mbox{as}\ n \to \infty.
\end{align*}
The first statement follows from Lemma \ref{OL4} iv) which also yields
a uniform bound for $\bfw_{m,n}$ as a consequence of \eqref{AlIn}.
The convergence with respect to the metric $D$ follows from Lemma \ref{OL4} iii), the uniform bounds for $\bfw_{m,n}$ and dominated convergence.
Moreover, due to Lemma \ref{OL4} iii), we have
\[
\begin{split}
&\liminf_{n \to \infty} I[\vc{v}_m + \vc{w}_{m,n}]= I[\vc{v}_m] + \liminf_{n \to \infty} \frac{1}{2} \expe{ \int_0^T \intTN{
\frac{ |\vc{w}_{m,n}|^2 }{r} } \dt }.
\end{split}
\]
Here, by virtue of (\ref{oscil}), Fatou's lemma and Jensen's inequality
\[
\begin{split}
\liminf_{n \to \infty} \frac{1}{2} \expe{ \int_0^T \intTN{
\frac{ |\vc{w}_{m,n}|^2 }{r} } \dt } &\geq \frac{c(N,T)}{e} \left( \expe{ \int_0^T \intTN{ \left[e - \frac{1}{2} \frac{|\vc{v}_m + \vc{h}|^2 }{r} \right] } \dt } \right)^2
\\&= \frac{c(N,T)}{e} I^2[\vc{v}_m] \geq \ep^2 \frac{c(N,T)}{e}.
\end{split}
\]
In such a way, we may construct a sequence $(\tilde{\vc{v}}_k) \subset X_0$, $D[\tilde{\vc{v}}_k, \vc{v}] \to 0$, and
\begin{align}\label{eq:27.04}
\liminf_{k \to \infty} I[\tilde{\vc{v}}_k] > I[\vc{v}].
\end{align}
For instance, we could take $\tilde{\vc{v}}_k=\vc{v}_k+\vc{w}_{k,k}$.
Relation \eqref{eq:27.04} contradicts the assumption that $\vc{v}$ is a point of continuity of $I$.

%

\def\cprime{$'$} \def\ocirc#1{\ifmmode\setbox0=\hbox{$#1$}\dimen0=\ht0
  \advance\dimen0 by1pt\rlap{\hbox to\wd0{\hss\raise\dimen0
  \hbox{\hskip.2em$\scriptscriptstyle\circ$}\hss}}#1\else {\accent"17 #1}\fi}

\end{document}